\documentclass[10pt]{amsart}

\usepackage[T1]{fontenc}
\usepackage{lmodern}
\usepackage{amsmath}
\usepackage{amssymb}
\usepackage{amsthm}
\usepackage{mathtools}
\usepackage{mathrsfs}
\usepackage{xcolor}
\usepackage{tikz}
\PassOptionsToPackage{hyphens}{url}
\usepackage{hyperref}
\usepackage{microtype}
\usepackage[textwidth=400pt,textheight=640pt,hcentering,heightrounded]{geometry}

\hypersetup{
  colorlinks=true,
  linkcolor=blue!55!black,
  citecolor=green!40!black,
  urlcolor=blue!65!black
}
\mathtoolsset{showonlyrefs=true}
\allowdisplaybreaks
\newtheorem{theorem}{Theorem}[section]
\newtheorem{proposition}[theorem]{Proposition}
\newtheorem{lemma}[theorem]{Lemma}
\newtheorem{corollary}[theorem]{Corollary}

\theoremstyle{definition}

\theoremstyle{remark}
\newtheorem{remark}[theorem]{Remark}
\newtheorem{fact}[theorem]{Fact}

\makeatletter
\def\subsection{\@startsection{subsection}{2}%
  \z@{.5\linespacing\@plus.7\linespacing}{.3\linespacing}%
  {\normalfont\bfseries}}
\makeatother

\newcommand{\C}{\mathbb C}
\newcommand{\Q}{\mathbb Q}
\newcommand{\Z}{\mathbb Z}
\newcommand{\N}{\mathbb N}
\newcommand{\R}{\mathbb R}
\newcommand{\bbeta}{\boldsymbol{\beta}}
\newcommand{\bv}{\boldsymbol{v}}
\newcommand{\bu}{\boldsymbol{u}}
\newcommand{\bg}{\boldsymbol{g}}
\newcommand{\bh}{\boldsymbol{h}}
\newcommand{\bs}{\boldsymbol{s}}
\newcommand{\bx}{\boldsymbol{x}}
\newcommand{\bpartial}{\boldsymbol{\partial}}
\newcommand{\nsupp}{\operatorname{nsupp}}
\newcommand{\inw}{\operatorname{in}_{\boldsymbol w}}
\newcommand{\Span}{\operatorname{Span}_{\C}}
\newcommand{\rank}{\operatorname{rank}}
\newcommand{\gr}{\operatorname{gr}}
\newcommand{\tp}[1]{\prescript{t}{}{#1}}
\newcommand{\Rhat}{\widehat R}
\newcommand{\Shat}{\widehat S}
\newcommand{\cB}{\mathcal B}
\newcommand{\cC}{\mathcal C}
\newcommand{\cD}{\mathcal D}
\newcommand{\cE}{\mathcal E}
\newcommand{\cF}{\mathcal F}
\newcommand{\cG}{\mathcal G}
\newcommand{\cH}{\mathcal H}
\newcommand{\cI}{\mathcal I}
\newcommand{\cJ}{\mathcal J}
\newcommand{\cN}{\mathcal N}
\newcommand{\cP}{\mathcal P}
\newcommand{\cS}{\mathcal S}
\newcommand{\cV}{\mathcal V}
\newcommand{\cW}{\mathcal W}
\newcommand{\cQ}{\mathcal Q}
\newcommand{\cT}{\mathcal T}
\newcommand{\cU}{\mathcal U}
\newcommand{\cX}{\mathcal X}
\newcommand{\sE}{\mathscr E}
\newcommand{\sF}{\mathscr F}
\newcommand{\sO}{\mathscr O}
\newcommand{\sP}{\mathscr P}
\newcommand{\sS}{\mathscr S}
\newcommand{\fD}{\mathfrak D}
\newcommand{\fm}{\mathfrak m}
\newcommand{\fs}{\mathfrak s}
\newcommand{\ba}{\boldsymbol{a}}
\newcommand{\bb}{\boldsymbol{b}}
\newcommand{\bp}{\boldsymbol{p}}
\newcommand{\br}{\boldsymbol{r}}
\newcommand{\bt}{\boldsymbol{t}}
\newcommand{\bw}{\boldsymbol{w}}
\newcommand{\by}{\boldsymbol{y}}
\newcommand{\bz}{\boldsymbol{z}}
\newcommand{\balpha}{\boldsymbol{\alpha}}
\newcommand{\bgamma}{\boldsymbol{\gamma}}
\newcommand{\bxi}{\boldsymbol{\xi}}
\newcommand{\bzeta}{\boldsymbol{\zeta}}
\newcommand{\bk}{\boldsymbol{k}}
\newcommand{\bzero}{\boldsymbol{0}}
\newcommand{\bone}{\boldsymbol{1}}
\newcommand{\Sol}{\operatorname{Sol}}
\newcommand{\im}{\operatorname{im}}
\newcommand{\supp}{\operatorname{supp}}
\newcommand{\Sing}{\operatorname{Sing}}
\newcommand{\pr}{\operatorname{pr}}
\newcommand{\amb}{\mathrm{amb}}
\newcommand{\intr}{\mathrm{intr}}
\newcommand{\ord}{\mathrm{ord}}

\hypersetup{
  pdftitle={A codimension formula for intrinsic perturbation of A-hypergeometric series},
  pdfauthor={NAKANO Ryunosuke},
  pdfsubject={A codimension formula for intrinsic perturbation of A-hypergeometric series. Mathematics Subject Classification 2020: 33C70 (Primary); 13D07, 13P10, 13F55, 14M25 (Secondary)},
  pdfkeywords={A-hypergeometric system, Frobenius method, intrinsic perturbation, fake exponent, Nilsson series, Groebner basis, colon ideal}
}

\title[A codimension formula for intrinsic perturbation]{A codimension formula
for intrinsic perturbation of \(A\)-hypergeometric series}
\author{NAKANO Ryunosuke}
\address{Graduate School of Science, Hokkaido University, Sapporo 060-0810, Japan}
\email{nakano.ryunosuke.i3@elms.hokudai.ac.jp}
\date{}

\subjclass[2020]{Primary 33C70; Secondary 13D07, 13P10, 13F55, 14M25}
\keywords{\(A\)-hypergeometric system, Frobenius method, intrinsic perturbation, fake exponent, Nilsson series, Gr\"obner basis, colon ideal}

\begin{document}

\begin{abstract}
	We study the canonical formal solution space of a homogeneous \(A\)-hypergeometric system in a generic direction.
	We present as a finite-dimensional cokernel the quotient of that space by the span of the series obtained by intrinsic perturbation, taken over all exponents occurring in that space and all ordered negative support families, and we compute the codimension of that span.
	We show that the presentation and the codimension transfer to the holomorphic solutions on a common nonsingular domain.
	We apply these results, for every integer \(q\geq1\), to a homogeneous \(A\)-hypergeometric system of lattice rank 2 with \(5q\) columns.
	We show that the canonical formal solution space of this system has dimension \(12q^2\) and that the span of the series obtained by intrinsic perturbation has codimension exactly \(\lceil3q^2/4\rceil\) in that space.
\end{abstract}

\maketitle

\section{Introduction}

The Frobenius method produces logarithmic solutions of a resonant differential system by perturbing an exponent and differentiating with respect to the resulting perturbation parameters.
For \(A\)-hypergeometric systems, this method is developed by Saito and by Okuyama--Saito; see \cite{Sai20,OS22}.
Let \(A\in\Z^{d\times n}\) be the matrix of the system, put \(L=\ker_{\Z}(A)\) and \(r=\rank L\), and write \(\bt=\tp{(t_1,\ldots,t_n)}\).
Okuyama--Saito \cite{OS25} fix a fake exponent \(\bv\) and an ordered negative support family \(\cN\), construct the ambient coefficient space, and compare it with the coefficient space obtained by intrinsic perturbation within \(L\); their Theorem~5.5 identifies the ambient coefficient space with the full coefficient space, and their Proposition~6.2 characterizes the equality of the intrinsic and ambient coefficient spaces by an equality of two colon ideals.
The companion paper \cite{Nak26} measures by an intrinsic-perturbation obstruction module the failure of the intrinsic and ambient coefficient spaces to agree at a fixed fake exponent, and gives configurations for which this module is nonzero.
However, the obstruction at a fixed fake exponent does not determine how much of the solution space the series obtained by intrinsic perturbation span.
In this paper, we assemble the local obstructions along each \(L\)-coset and compute the codimension of the span of the series obtained by intrinsic perturbation in the canonical formal solution space.

We work in \(\Rhat=\C[[t_1,\ldots,t_n]]\) with the ideal \(U=\langle A\bt\rangle\), the monomial ideals \(M_{\cN}\) and \(P_{\cN}(\bt)\) of \eqref{eq:obstruction-M} and \eqref{eq:obstruction-P}, and the monomial \(e=\bt^{\nsupp(\bv)\setminus K_{\cN}}\), where \(\nsupp(\bv)\) is the set of indices \(j\) with \(v_j\) a negative integer and \(K_{\cN}\) is the intersection of the members of \(\cN\).
The intrinsic-perturbation obstruction module introduced in \cite{Nak26} is the quotient
\[
	\fD_{\cN}(e)
	=\frac{(U+P_{\cN}(\bt)):e}
	{(UM_{\cN}+P_{\cN}(\bt)):e}.
\]
Fact~\ref{fact:finite-presentation} presents this module by two ideals \(J_{\cN}^{\amb}\subseteq J_{\cN}^{\intr}\) of a polynomial ring in \(r\) variables, and Fact~\ref{fact:defect-coefficient-duality} identifies the graded dual of their quotient with the ambient coefficient space modulo the intrinsic one.

For a homogeneous system and a generic direction, Theorem~\ref{thm:formal-solution-exact-cokernel} presents, as the cokernel of a block-triangular map between finite-dimensional spaces, the quotient of the canonical formal solution space by the span of the series obtained by intrinsic perturbation.
Here the span is taken over all exponents occurring in that space and all ordered negative support families.
Corollary~\ref{cor:formal-solution-exact-codimension} computes the codimension of that span as the sum of the local codimensions at the exponents, less a correction term for each \(L\)-coset.
Theorem~\ref{thm:analytic-realization} shows that summation on a nonempty simply connected nonsingular open set, for a choice of the branches of the logarithms, carries this presentation and this codimension to the holomorphic solution space.
For a direction that is nonzero on every nonzero element of \(L\), Okuyama \cite[Theorems~3.6, 5.4, and~6.2]{Oku26} presents directed formal logarithmic solution spaces over Euler fibers by finite coefficient modules without a homogeneity assumption.
Theorems~6.4 and~6.6 of \cite{Oku26} give finite criteria for realizable lowest exponents using kernels of coefficient maps and colon ideals.
Here we compute the codimension of the span of the series obtained by intrinsic perturbation in the canonical formal solution space under the standing homogeneity and genericity assumptions.

We apply these results to two of the configurations of \cite{Nak26}.
For the five-column configuration of Theorem~5.1 of \cite{Nak26}, Corollary~\ref{cor:lattice-counterexample-formal-codimension} gives codimension one.
For the family of lattice rank 2 of Theorem~6.2 of \cite{Nak26}, we show in Theorem~\ref{thm:fixed-rank-unbounded-solution-codimension} that, for every integer \(q\geq1\), the canonical formal solution space has dimension \(12q^2\) and the span of the series obtained by intrinsic perturbation has codimension exactly \(\lceil3q^2/4\rceil\) in that space.
The same dimensions hold on a common nonsingular domain, and the codimension is unbounded at fixed lattice rank.
The ancillary scripts use exact arithmetic to check the computations giving codimension one in the canonical formal solution space for the five-column example and the local codimensions obtained from all ordered negative support families for \(q=1,2,3,4\) in the family in Section~\ref{sec:applications}.

\section{Notation and the perturbation constructions}
\label{sec:fake-exponents}

We recall the objects that the codimension formulas of this paper use, and we fix notation for them, following \cite{GKZ89,GKZ94,Sai02,SST00} and \cite{OS25}.

\subsection{The system and a generic direction}

Throughout this paper, \(A=[\ba_1,\ldots,\ba_n]\in\Z^{d\times n}\) is a matrix of rank \(d\) whose columns lie in an affine hyperplane that does not contain the origin.
The lattice of relations among the columns is \(L=\ker_{\Z}(A)\), and we put \(r=\rank L=n-d\) and \(\N=\{0,1,2,\ldots\}\).
Every \(\bu\in L\) is written as \(\bu=\bu_+-\bu_-\), where \(\bu_+\) and \(\bu_-\) lie in \(\N^n\) and have disjoint supports, and we put \(|\bu|=\sum_{j=1}^n u_j\).

The toric ideal of \(A\) is \(I_A= \left\langle \bpartial^{\bu_+} -\bpartial^{\bu_-} \ \middle|\ \bu\in L \right\rangle \subseteq\C[\partial_1,\ldots,\partial_n]\).
We write \(D\) for the Weyl algebra in \(\bx=\tp{(x_1,\ldots,x_n)}\) and \(\bpartial=\tp{(\partial_1,\ldots,\partial_n)}\).
For \(\bbeta\in\C^d\), the \(A\)-hypergeometric ideal \(H_A(\bbeta)\) is generated in \(D\) by \(I_A\) together with the Euler operators \(E_i-\beta_i =\sum_{j=1}^n a_{ij}x_j\partial_j-\beta_i\) for \(i=1,\ldots,d\), and we put \(M_A(\bbeta)=D/H_A(\bbeta)\).
This system is introduced in \cite{GKZ89,GKZ94}, and the theory of fake exponents recalled below is developed in \cite{Sai02,SST00}.

Throughout this paper, \(\bw\in\R^n\) denotes a generic weight.
By genericity we mean that the initial ideal \(\inw(I_A)\) is monomial and that \(\bw\) lies in the interior of a full-dimensional cone of the small Gr\"obner fan of \(H_A(\bbeta)\), so that the initial ideal \(\operatorname{in}_{(-\bw,\bw)}(H_A(\bbeta))\) in the Weyl algebra is constant on that interior; see \cite{SST00}.
Weights of this kind occur in every full-dimensional Gr\"obner cone of \(I_A\).
A wall of either of these fans is the locus on which two monomials of the same \(A\)-degree acquire equal \(\bw\)-weight, and the difference of the two exponents is a nonzero element of \(L\).
It follows that a weight is generic whenever its values on the nonzero elements of \(L\) are all nonzero, because such a weight lies on no wall.

Let \(\cG_{\bw} = \left\{ \bpartial^{\bg^{(i)}_+} -\bpartial^{\bg^{(i)}_-} \ \middle|\ i=1,\ldots,g_{\bw} \right\}\) be the reduced Gr\"obner basis of \(I_A\), in which each binomial is written with its initial monomial first.
Put \(\bg^{(i)}=\bg^{(i)}_+-\bg^{(i)}_-\), so that \(\bw\cdot\bg^{(i)}>0\) for \(i=1,\ldots,g_{\bw}\).
These vectors generate the affine semigroup \(\cC(\bw) =\sum_{i=1}^{g_{\bw}}\N\bg^{(i)}\subseteq L\).
We say that \(\cC(\bw)\) is normal if it contains every element of the group generated by \(\cC(\bw)\) that lies in the real cone spanned by \(\cC(\bw)\).
Since \(\cC(\bw)\) is contained in \(L\), this normality is a property of the relations among the columns of \(A\), not of the affine semigroup that the columns generate.
By the column matroid of \(A\) we mean the matroid that the columns of \(A\) represent over \(\Q\).

For \(\bv\in\C^n\), put \(\nsupp(\bv)=\{j\mid v_j\in\Z_{<0}\}\).
Following \cite[Section~3.2]{SST00}, we call \(\bv\) a fake exponent of \(H_A(\bbeta)\) in the direction \(\bw\) when \(A\bv=\bbeta\) and \(\bv\) is a zero of the fake indicial ideal.
Corollary~3.2.3 of \cite{SST00} states that a vector \(\bv\) with \(A\bv=\bbeta\) is a fake exponent if and only if some standard pair \((\ba,\sigma)\) of \(\inw(I_A)\) satisfies \(v_j=a_j\) for every \(j\notin\sigma\).
We use this description of the fake exponents throughout.

\subsection{Negative supports and their families}

Fix a fake exponent \(\bv\) for the remainder of this section.
Each \(\bu\in L\) determines the set \(I_{\bu}=\nsupp(\bv+\bu)\); in particular \(I_{\bzero}=\nsupp(\bv)\), and these sets form \(\sS(\bv)=\{I_{\bu}\mid\bu\in L\}\).
For \(I\in\sS(\bv)\), the support fiber of \(I\) is \(\cF_I(\bv)=\{\bu\in L\mid I_{\bu}=I\}\), and we shorten this notation to \(\cF_I\) because \(\bv\) is fixed.
The distinguished collection of \cite[Section~3]{OS25} consists of those \(I\in\sS(\bv)\) on whose support fiber \(\bw\) takes only nonnegative values, and we write
\[
	\cN_{\bv}
	=
	\left\{
	I\in\sS(\bv)
	\ \middle|\
	\bw\cdot\bu\geq0
	\text{ for every }\bu\in\cF_I
	\right\}.
\]
Since \(\bv\) is a fake exponent, the set \(I_{\bzero}\) lies in \(\cN_{\bv}\), as the discussion preceding \cite[Definition~3.1]{OS25} shows.
A negative support family for \(\bv\) is a subset \(\cN\) with \(I_{\bzero}\in\cN\subseteq\cN_{\bv}\), and we put \(K_{\cN}=\bigcap_{I\in\cN}I\).
Such a family is called ordered in \cite[Definition~3.1]{OS25} when every \(J\in\sS(\bv)\) contained in some \(I\in\cN\) again lies in \(\cN\).
The distinguished collection \(\cN_{\bv}\) is nowhere assumed to be ordered.

\subsection{Perturbation series and coefficient spaces}

We now recall the intrinsic perturbation series and the three spaces of coefficients of \(\bx^{\bv}\) that the constructions of \cite{OS25} produce.
Choose a Gale dual \(B=(\bb^{(1)},\ldots,\bb^{(r)})\) of \(A\) whose columns form a \(\Z\)-basis of \(L\), and put \(\bs=\tp{(s_1,\ldots,s_r)}\).
For \(\bz\in\C^n\) and \(\bp\in\N^n\), the falling factorial is \([\bz]_{\bp} = \prod_{j=1}^n z_j(z_j-1)\cdots(z_j-p_j+1)\).
For \(\bu\in L\), set
\[
	a_{\bu}(\bs)
	=
	\frac{[\bv+B\bs]_{\bu_-}}
	{[\bv+B\bs+\bu]_{\bu_+}}.
\]
Given a negative support family \(\cN\), put \(m_{\bv,\cN}(\bs) = \prod_{j\in I_{\bzero}\setminus K_{\cN}}(B\bs)_j\) and define the intrinsic perturbation series by
\[
	\Psi_{\cN}(\bx,\bs)
	=
	m_{\bv,\cN}(\bs)
	\sum_{\substack{\bu\in L\\I_{\bu}\in\cN}}
	a_{\bu}(\bs)\bx^{\bv+B\bs+\bu}.
\]
Only the polynomial expression of \(m_{\bv,\cN}\) depends on the choice of \(B\), and the index set \(I_{\bzero}\setminus K_{\cN}\) does not.
A change of Gale dual to \(B'=BT\), with \(T\in\operatorname{GL}_r(\Z)\), transforms the expression for \(B\) into the expression for \(B'\) by the substitution \(f(\bs)\mapsto f(T\bs)\).

The full coefficient space at \(\bv\) for \(\cN\) is the space of logarithmic polynomials that arise as the coefficient of \(\bx^{\bv}\) in those formal series solutions of \(H_A(\bbeta)\) in the direction \(\bw\) all of whose negative supports \(I_{\bu}\) lie in \(\cN\).
If \(\cN\) is ordered, the intrinsic coefficient space at \(\bv\) for \(\cN\) is the space of coefficients of \(\bx^{\bv}\) that the construction of \cite[Theorem~3.2]{OS25} produces from \(\Psi_{\cN}(\bx,\bs)\).
Under the same hypothesis on \(\cN\), the ambient coefficient space at \(\bv\) for \(\cN\) is the space of coefficients of \(\bx^{\bv}\) that the ambient perturbation construction of \cite[Theorem~5.5]{OS25} produces.

For each \(i\), put \(G^{(i)}(\bv) =I_{-\bg^{(i)}}\setminus I_{\bzero}\), and define
\[
	P_B(\bv)
	=
	\left\langle
	\prod_{j\in G^{(i)}(\bv)}(B\bs)_j
	\ \middle|\
	i=1,\ldots,g_{\bw}
	\right\rangle
	\subseteq\C[\bs].
\]
For \(f(\bs)\in\C[\bs]\) and \(q(\bpartial_{\bs})\in\C[\bpartial_{\bs}]\), put
\[
	\left\langle f,q(\bpartial_{\bs})\right\rangle_{\bs}
	=
	\left.q(\bpartial_{\bs})f(\bs)\right|_{\bs=\bzero}.
\]
For a homogeneous ideal \(I\subseteq\C[\bs]\), this pairing defines the inverse system
\[
	I^\perp
	=
	\left\{
	q(\bpartial_{\bs})
	\ \middle|\
	\left\langle f,q(\bpartial_{\bs})\right\rangle_{\bs}=0
	\text{ for every }f\in I
	\right\}.
\]
We use this construction for \(P_B(\bv)\) and, in Section~\ref{subsec:finite-boundary-presentation}, for the ideals \(J_{\cN}^{\amb}\) and \(J_{\cN}^{\intr}\).

The negative support families, the perturbation constructions, and the ideal \(P_B(\bv)\) are taken from \cite{Sai20,OS22,OS25}.
The affine semigroup \(\cC(\bw)\) also occurs in \cite{Nag26}.

\begin{remark}
	\label{rem:homogeneity-scope}
	The standing assumption on the columns of \(A\) holds if and only if some linear functional \(\bh\) satisfies \(\bh\ba_j=1\) for every \(j\), and this assumption makes \(I_A\) homogeneous for the standard grading.
	The assumption is used where a weight is moved to \(\bw+\lambda\bone\), with \(\bone\) the all-ones vector and \(\lambda\) a positive number, without changing the weight of any lattice vector.
	That move requires \(\bone\) to lie in the row space of \(A\), and the move underlies the canonical formal solution space and the results in Corollary~\ref{cor:formal-solution-exact-codimension} and Theorem~\ref{thm:analytic-realization}.
	The assumption also enters through the results of \cite{OS25}, which are stated for a homogeneous configuration and supply the identification of the ambient coefficient space with the full coefficient space.
	We therefore impose the assumption throughout instead of tracking it statement by statement.
\end{remark}

\section{The obstruction to intrinsic perturbation}
\label{sec:lattice-obstruction}

Ambient perturbation is not restricted to directions in \(L\), and by \cite[Theorem~5.5]{OS25} it realizes the full coefficient space for an ordered negative support family.
Intrinsic perturbation uses only directions in \(L\) and gives a subspace of that coefficient space.
In this section, we recall from \cite{Nak26} the quotient of the two colon ideals of \cite[Proposition~6.2]{OS25} and the two ideals in \(r\) variables that present it.

Table~\ref{tab:obstruction-notation} collects the notation of this section.
\begin{table}[ht]
	\centering
	\small
	\begin{tabular}{@{}p{.29\linewidth}p{.65\linewidth}@{}}
		\(\sS(\bv)\)                            &
		the negative supports \(I_{\bu}\) that occur for \(\bu\in L\)                               \\
		\(\cN_{\bv}\)                           &
		the distinguished collection of \cite[Section~3]{OS25}                                      \\
		\(\cN\)                                 &
		a negative support family for \(\bv\)                                                       \\
		\(\Rhat\), \(R_0\)                      &
		\(\C[[t_1,\ldots,t_n]]\) and \(\C[t_1,\ldots,t_n]\), respectively                           \\
		\(\Shat\), \(S_0\)                      &
		\(\C[[s_1,\ldots,s_r]]\) and \(\C[s_1,\ldots,s_r]\), respectively                           \\
		\(K_{\cN}\), \(E\), \(e\)               &
		\(\bigcap_{I\in\cN}I\),
		\(I_{\bzero}\setminus K_{\cN}\), and
		\(\bt^E\), respectively                                                                     \\
		\(M_{\cN}\), \(P_{\cN}(\bt)\)           &
		the monomial ideals in \eqref{eq:obstruction-M} and \eqref{eq:obstruction-P}, respectively  \\
		\(Q_{\cN}(\bt)\)                        &
		\(UM_{\cN}+P_{\cN}(\bt)\)                                                                   \\
		\(\cH_{\cN}\)                           &
		the antichain in \eqref{eq:boundary-H-antichain}                                            \\
		\(\fD_{\cN}(e)\)                        &
		the intrinsic-perturbation obstruction module in \eqref{eq:defect-module}                   \\
		\(J_{\cN}^{\amb}\), \(J_{\cN}^{\intr}\) &
		the two ideals in \eqref{eq:boundary-comparison-ideals}
	\end{tabular}
	\caption{Notation for the obstruction module and its presentation.}
	\label{tab:obstruction-notation}
\end{table}

\subsection{The obstruction module}

We keep fixed the fake exponent \(\bv\) and the negative support family \(\cN\) of Section~\ref{sec:fake-exponents}.
None of the constructions below requires \(\cN\) to be ordered, and that hypothesis enters only through the interpretation in terms of the ambient and intrinsic coefficient spaces.
Write \(\bt^J=\prod_{j\in J}t_j\), with \(\bt^{\varnothing}=1\), and put \(e=\bt^{I_{\bzero}\setminus K_{\cN}}\).
Inside the complete local ring \(\Rhat=\C[[t_1,\ldots,t_n]]\), with \(\bt=\tp{(t_1,\ldots,t_n)}\), define
\begin{align}
	U&=\langle A\bt\rangle,
	\label{eq:obstruction-U}\\
	M_{\cN}
	&=\left\langle
	\bt^{I\setminus K_{\cN}}
	\ \middle|\ I\in\cN
	\right\rangle,
	\label{eq:obstruction-M}\\
	P_{\cN}(\bt)
	&=\left\langle
	\bt^{(I\cup J)\setminus K_{\cN}}
	\ \middle|\
	I\in\cN,\ J\in\sS(\bv)\setminus\cN
	\right\rangle,
	\label{eq:obstruction-P}\\
	Q_{\cN}(\bt)
	&=UM_{\cN}+P_{\cN}(\bt).
	\label{eq:obstruction-Q}
\end{align}
Let \(\Shat=\C[[s_1,\ldots,s_r]]\), and consider the surjection
\[
	\Phi:\Rhat\longrightarrow\Shat, \qquad t_j\longmapsto(B\bs)_j,
\]
whose kernel is \(U\).
Write \(\Pi_{\cN}=\Phi(P_{\cN}(\bt))\) and \(m_{\bv,\cN}=\Phi(e)\).

\begin{fact}
	\label{fact:ambient-realization}
	Let \(A\) be homogeneous, let \(\bw\) be generic, let \(\bv\) be a fake exponent of \(H_A(\bbeta)\) in the direction \(\bw\), and let \(\cN\) be an ordered negative support family for \(\bv\).
	By \cite[Theorem~5.5]{OS25}, each series obtained by ambient perturbation is a formal series solution of \(H_A(\bbeta)\) in the direction \(\bw\), and the coefficients of \(\bx^{\bv}\) in these series are exactly the elements of the full coefficient space at \(\bv\) for \(\cN\).
\end{fact}

In other words, the ambient coefficient space and the full coefficient space at \(\bv\) for \(\cN\) are equal.
By \cite[Proposition~6.2]{OS25}, the intrinsic coefficient space at \(\bv\) for \(\cN\) lies in the ambient coefficient space, and equality of the two holds if and only if \(\Phi\bigl(Q_{\cN}(\bt):e\bigr) = \Pi_{\cN}:m_{\bv,\cN}\).
Following \cite{Nak26}, put
\begin{equation}
	\fD_{\cN}(e)
	=
	\frac{(U+P_{\cN}(\bt)):e}
	{(UM_{\cN}+P_{\cN}(\bt)):e}.
	\label{eq:defect-module}
\end{equation}
We impose no finite-length hypothesis on \(\fD_{\cN}(e)\), and we write \(\dim_{\C}\fD_{\cN}(e)\) for the dimension of this module when it has finite length.

\subsection{A finite presentation of the two colon ideals}
\label{subsec:finite-boundary-presentation}

The presentation recalled here replaces the complete local rings by the polynomial rings \(R_0=\C[t_1,\ldots,t_n]\) and \(S_0=\C[s_1,\ldots,s_r]\), and we put \(\fm=\langle s_1,\ldots,s_r\rangle\).
Put \(\theta_i=\sum_{j=1}^n a_{ij}t_j\) for \(1\leq i\leq d\), and define
\[
	U_0=\langle\theta_1,\ldots,\theta_d\rangle,
	\qquad
	\Phi_0:R_0\longrightarrow S_0,
	\qquad
	t_j\longmapsto\ell_j=(B\bs)_j.
\]
It follows that \(\ker\Phi_0=U_0\).
Let \(M_{\cN}^0\) and \(P_{\cN}^0(\bt)\) be the monomial ideals of \(R_0\) generated by the monomials in \eqref{eq:obstruction-M} and \eqref{eq:obstruction-P}, and put \(Q_{\cN}^0(\bt)=U_0M_{\cN}^0+P_{\cN}^0(\bt)\).
For \(Y\subseteq\{1,\ldots,n\}\), write \(\ell^Y=\prod_{j\in Y}\ell_j\), with \(\ell^{\varnothing}=1\).

Below we use the antichain
\begin{equation}
	\cH_{\cN}
	=
	\min_{\subseteq}
	\{(I\cup J)\setminus K_{\cN}
	\mid I\in\cN,\
	J\in\sS(\bv)\setminus\cN\}
	=\{H_1,\ldots,H_h\},
	\label{eq:boundary-H-antichain}
\end{equation}
which is empty when \(\sS(\bv)=\cN\).
Put \(E=I_{\bzero}\setminus K_{\cN}\).
Each \(j\in E\) lies in some \(I\in\cN\) and outside some \(I'\in\cN\), so the \(j\)-th row of \(B\) and the linear form \(\ell_j\) are nonzero.
Hence \(\ell^E\ne0\).

Following \cite{Nak26}, define two homogeneous ideals of \(S_0\) by
\begin{equation}
	J_{\cN}^{\amb}=\Phi_0\bigl(Q_{\cN}^0(\bt):e\bigr), \qquad J_{\cN}^{\intr}=\Phi_0\bigl(P_{\cN}^0(\bt)\bigr):\Phi_0(e).
	\label{eq:boundary-comparison-ideals}
\end{equation}

\begin{fact}
	\label{fact:finite-presentation}
	Let \(\bv\) be a fake exponent and let \(\cN\) be a negative support family for \(\bv\).
	By Theorem~3.7 of \cite{Nak26}, the two ideals in \eqref{eq:boundary-comparison-ideals} satisfy
	\begin{equation}
		J_{\cN}^{\amb}
		\subseteq
		J_{\cN}^{\intr},
		\label{eq:boundary-annihilator-colon}
	\end{equation}
	their extensions to \(\Shat\) are \(\Phi(Q_{\cN}(\bt):e)\) and \(\Pi_{\cN}:m_{\bv,\cN}\), respectively, and there is an isomorphism
	\begin{equation}
		\fD_{\cN}(e)
		\simeq
		\Shat\otimes_{S_0}
		\frac{J_{\cN}^{\intr}}{J_{\cN}^{\amb}}.
		\label{eq:boundary-classification-isomorphism}
	\end{equation}
\end{fact}

The quotient \(J_{\cN}^{\intr}/J_{\cN}^{\amb}\) in \eqref{eq:boundary-classification-isomorphism} is a graded \(S_0\)-module, and \(\fD_{\cN}(e)\) is its completion; we do not regard the completed module as the direct sum of the homogeneous components of the quotient.

Each homogeneous subspace \(W\subseteq S_{0,i}\) has an orthogonal complement \(W^\perp\subseteq\C[\bpartial_{\bs}]_i\) with respect to \(\langle\cdot,\cdot\rangle_{\bs}\).
Put \(V_{\cN,i}^{\amb} = (J_{\cN}^{\amb})_i^\perp\) and \(V_{\cN,i}^{\intr} = (J_{\cN}^{\intr})_i^\perp\).

\begin{fact}
	\label{fact:defect-coefficient-duality}
	Let \(\bv\) and \(\cN\) be as in Fact~\ref{fact:finite-presentation}, and let \(i\geq0\).
	By Proposition~3.8 of \cite{Nak26}, there is an inclusion \(V_{\cN,i}^{\intr} \subseteq V_{\cN,i}^{\amb}\), and \(\langle\cdot,\cdot\rangle_{\bs}\) induces a perfect pairing
	\begin{equation}
		\left(
		\frac{J_{\cN}^{\intr}}
		{J_{\cN}^{\amb}}
		\right)_i
		\times
		\frac{V_{\cN,i}^{\amb}}
		{V_{\cN,i}^{\intr}}
		\longrightarrow\C.
		\label{eq:defect-coefficient-perfect-pairing}
	\end{equation}
	If \(\cN\) is ordered, Theorem~5.5 and Proposition~6.2 of \cite{OS25} identify \(V_{\cN,i}^{\amb}\) and \(V_{\cN,i}^{\intr}\) with the degree-\(i\) ambient and intrinsic coefficient spaces at \(\bv\), respectively.
\end{fact}

\subsection{Formal series solutions}

We now allow the fake exponent and the ordered negative support family to vary.
When the fake exponent must be displayed in the notation, write \(J_{\bv,\cN}^{\amb}\), \(J_{\bv,\cN}^{\intr}\), \(V_{\bv,\cN,i}^{\amb}\), and \(V_{\bv,\cN,i}^{\intr}\) for the objects defined above.
Define \(V_{\bv,\cN}^{\amb}\) and \(V_{\bv,\cN}^{\intr}\) by
\[
	V_{\bv,\cN}^{\amb}
	=
	\bigoplus_{i\geq0}V_{\bv,\cN,i}^{\amb},
	\qquad
	V_{\bv,\cN}^{\intr}
	=
	\bigoplus_{i\geq0}V_{\bv,\cN,i}^{\intr}.
\]
We identify these spaces with the corresponding spaces of logarithmic coefficients by the injective substitution in lattice coordinates given in \cite[Sections~5 and~6]{OS25}.

For a fake exponent \(\bv\), consider the elements of the canonical formal solution space in the direction \(\bw\) that are supported in \(\bv+L\) and whose terms have nonnegative \(\bw\)-weight relative to \(\bv\).
Define \(\cE_{\bv}\) to be the graded space of the logarithmic coefficients of \(\bx^{\bv}\) in those elements, written in lattice coordinates.
By \cite[Theorem~4.9]{OS25}, the space \(\cE_{\bv}\) is the finite-dimensional inverse system of the component supported at \(\bv\) of the indicial ideal \(\operatorname{ind}_{\bw}(H_A(\bbeta))\).
Here the inverse system is written in the lattice coordinates supplied by \(B\).
Let \(\cE_{\bv}^{\perp}\) be the orthogonal complement of \(\cE_{\bv}\) in \(S_0\) under \(\langle\cdot,\cdot\rangle_{\bs}\).
For each \(i\), the factors of the distraction of \(\bpartial^{\bg^{(i)}_+}\) that vanish at \(\bv\) are exactly the factors indexed by \(G^{(i)}(\bv)\).
The remaining factors are units in the local ring \(S_{0,\fm}\).
Thus \(P_B(\bv)\) is the corresponding shifted local component of the fake indicial ideal; see \cite[Section~3.2]{SST00}.
The defining inclusion of the fake indicial ideal in the indicial ideal gives
\begin{equation}
	P_B(\bv)
	\subseteq
	\cE_{\bv}^{\perp},
	\qquad
	\cE_{\bv}
	\subseteq
	P_B(\bv)^\perp.
	\label{eq:actual-fake-local-inclusions}
\end{equation}

We use the following fact about the construction of canonical series.

\begin{fact}
	\label{fact:canonical-series-construction}
	Let \(A\) be homogeneous, let \(\bbeta\in\C^d\), let \(\bw\) be generic, fix a monomial order refining the \(\bw\)-weight, and let \(N\) be the corresponding space of formal series in \cite[Section~2.5]{SST00}.
	The discussion following \cite[Algorithm~2.5.8]{SST00} states that the indicial ideal \(\operatorname{ind}_{\bw}(H_A(\bbeta))\) is a Frobenius ideal of finite rank equal to \(\rank(H_A(\bbeta))\) and that its zeros are the exponents of \(H_A(\bbeta)\) with respect to \(\bw\).
	The same discussion and the proof of \cite[Lemma~2.5.10]{SST00} show that this indicial ideal and \(\operatorname{in}_{(-\bw,\bw)}(H_A(\bbeta))\) have the same solutions in \(N\).
	By \cite[Theorem~2.3.11]{SST00}, these common solutions are finite sums of terms \(\bx^{\bv}q(\log\bx)\), where \(\bv\) is an exponent and \(q\) belongs to the local inverse system of the indicial ideal at \(\bv\).
	By \cite[Lemma~2.5.9, Corollary~2.5.11, and Theorem~2.5.12]{SST00}, there are exactly \(\rank(H_A(\bbeta))\) starting monomials with respect to the fixed order, the number of starting monomials with a fixed exponent equals the multiplicity of that exponent, and each starting monomial \(\fs\) determines a unique canonical series \(\Xi_{\fs}\in N\) whose coefficient at \(\fs\) is one and in which no other starting monomial occurs.
	Consequently, the series \(\Xi_{\fs}\) form a basis of the solutions of \(H_A(\bbeta)\) in \(N\), and their initial series form a basis of the solutions of the initial system in \(N\).
\end{fact}

By the description of \(\cE_{\bv}\) as an inverse system and Fact~\ref{fact:canonical-series-construction}, a fake exponent \(\bv\) is an exponent of \(H_A(\bbeta)\) with respect to \(\bw\) in the sense of \cite{Sai20} if and only if \(\cE_{\bv}\ne0\).
Under the standing homogeneity assumption, these are also the exponents in the sense of \cite[Definition~2.9]{DMM12}.
Let \(\sE_{\bbeta,\bw}\) be this finite set of exponents.
For \(\bv\in\sE_{\bbeta,\bw}\), define \(\sO(\bv)\) to be the set of all ordered negative support families for \(\bv\).
The set \(\sO(\bv)\) is finite because the members of \(\sS(\bv)\) are subsets of \(\{1,\ldots,n\}\).
Define \(\cN_{\bv}^{\ord}\) by
\begin{equation}
	\cN_{\bv}^{\ord}
	=
	\bigcup_{\cN\in\sO(\bv)}\cN.
	\label{eq:largest-ordered-family}
\end{equation}

The following normalization is adapted from \cite[Lemmas~4.1 and~4.2, and Proposition~4.3]{OS25}.
We record the following extension in Lemma~\ref{lem:coefficient-transport}: the lemma applies to a formal solution in the direction \(\bw\) without fixing an ordered negative support family.

\begin{lemma}
	\label{lem:coefficient-transport}
	Let \(\bv\) be a fake exponent, and let \(\by=(y_1,\ldots,y_n)\) denote variables representing \(\log\bx\).
	Let \(\phi(\bx)=\sum_{\bu\in L}\bx^{\bv+\bu}r_{\bu}(\log\bx)\) be an element of the canonical formal solution space in the direction \(\bw\) whose occurring shifts satisfy \(\bw\cdot\bu\geq0\), where \(r_{\bu}\in\C[\by]\) and \(r_{\bu}=0\) when the shift \(\bu\) does not occur.
	There are normalized coefficients \(c_{\bu}\in\C[\by]\), for \(\bu\in L\), such that \(c_{\bu}\) depends only on \(I_{\bu}\) and, writing \(c_I=c_{\bu}\) for \(I=I_{\bu}\), we have
	\begin{equation}
		\bpartial_{\by}^{J\setminus I}c_I
		-
		\bpartial_{\by}^{I\setminus J}c_J
		=0
		\qquad
		(I,J\in\sS(\bv)).
		\label{eq:coefficient-transport}
	\end{equation}
\end{lemma}

\begin{proof}
	For \(\bu,\bu'\in L\), define \(d_{\bu'\leftarrow\bu}\) by
	\[
		d_{\bu'\leftarrow\bu} = \prod_{\nu=1}^{n} \prod_{\mu=1}^{u_{\nu}-u'_{\nu}} \left(\partial_{y_{\nu}}+v_{\nu}+u_{\nu}-\mu+1\right),
	\]
	where the product over \(\mu\) is \(1\) when \(u_{\nu}-u'_{\nu}\leq0\).
	With the convention that the coefficient of an absent shift is zero, \cite[Lemma~4.1]{OS25} gives
	\begin{equation}
		d_{\bu'\leftarrow\bu}r_{\bu}
		-
		d_{\bu\leftarrow\bu'}r_{\bu'}
		=0.
		\label{eq:coefficient-transport-unnormalized}
	\end{equation}
	Coordinatewise factorization gives a unique polynomial differential operator \(\widetilde d_{\bu'\leftarrow\bu}\in\C[\bpartial_{\by}]\) with a nonzero constant term such that
	\begin{equation}
		d_{\bu'\leftarrow\bu}
		=
		\widetilde d_{\bu'\leftarrow\bu}
		\bpartial_{\by}^{I_{\bu'}\setminus I_{\bu}}.
		\label{eq:coefficient-transport-factorization}
	\end{equation}
	By \cite[Lemma~4.2]{OS25}, this factorization exists and is unique, and \(\widetilde d_{\bu'\leftarrow\bu}\) is an automorphism of \(\C[\by]\).
	The inverse of \(\widetilde d_{\bu'\leftarrow\bu}\) terminates on each polynomial; no inverse on a larger formal or analytic coefficient space is used here.
	Define the normalized coefficient \(c_{\bu}\) by
	\begin{equation}
		c_{\bu}
		=
		\left(\widetilde d_{\bu\leftarrow\bzero}\right)^{-1}
		\widetilde d_{\bzero\leftarrow\bu}r_{\bu}.
		\label{eq:coefficient-transport-normalization}
	\end{equation}
	Thus \(c_{\bu}=0\) if and only if \(r_{\bu}=0\); this equivalence holds for every shift \(\bu\), including the shifts added by zero extension.
	The cocycle identity for the operators \(\widetilde d\) is \cite[Lemma~4.2]{OS25}, and \cite[Proposition~4.3]{OS25} deduces \eqref{eq:coefficient-transport} from it.
	If \(I_{\bu}=I_{\bu'}\), both monomial differential operators are \(1\), and hence \(c_{\bu}=c_{\bu'}\).
	The normalized coefficient therefore depends only on the negative support.
	We have not assumed that a negative support family is ordered at any point in this proof.
\end{proof}

\begin{lemma}
	\label{lem:largest-ordered-family}
	For every \(\bv\in\sE_{\bbeta,\bw}\), the set \(\sO(\bv)\) is nonempty, and \(\cN_{\bv}^{\ord}\) is the largest ordered negative support family for \(\bv\).
	Moreover,
	\begin{equation}
		V_{\bv,\cN_{\bv}^{\ord}}^{\amb}
		=
		\cE_{\bv},
		\qquad
		J_{\bv,\cN_{\bv}^{\ord}}^{\amb}
		=
		\cE_{\bv}^{\perp}.
		\label{eq:largest-ordered-family-full-inverse-system}
	\end{equation}
\end{lemma}

\begin{proof}
	Take a nonzero element \(p\in\cE_{\bv}\).
	By the definition of \(\cE_{\bv}\), there is an element of the canonical formal solution space in the direction \(\bw\), supported in \(\bv+L\) with nonnegative relative weights, whose coefficient at \(\bx^{\bv}\) is \(p\).
	Apply Lemma~\ref{lem:coefficient-transport} after extending the coefficients of this series by zero to all shifts in \(L\), and let \((c_I)_{I\in\sS(\bv)}\) be the resulting normalized coefficients.
	The coefficient \(c_{I_{\bzero}}\) is \(p\), so it is nonzero.
	Define \(\cN(p)\) by \(\cN(p)=\{I\in\sS(\bv)\mid c_I\ne0\}\).
	The normalizing operators in Lemma~\ref{lem:coefficient-transport} are invertible, so \(\cN(p)\) is exactly the set of negative supports that occur in the chosen series.
	We use the proof of \cite[Lemma~4.8]{OS25}, which treats an \(A\)-hypergeometric series in the direction \(\bw\) without fixing an ordered negative support family.
	If \(c_{I_{\bu}}\ne0\) and \(I_{\bu}\notin\cN_{\bv}\), there is a shift \(\bu'\in L\) such that \(I_{\bu'}=I_{\bu}\) and \(\bw\cdot\bu'<0\).
	The nonnegative relative weights exclude the shift \(\bu'\), so the normalized coefficient at \(\bu'\) is zero.
	By Lemma~\ref{lem:coefficient-transport}, the normalized coefficient depends only on the negative support, so \(c_{I_{\bu}}=0\), which is a contradiction.
	Thus \(\cN(p)\subseteq\cN_{\bv}\), and \(c_{I_{\bzero}}=p\ne0\) gives \(I_{\bzero}\in\cN(p)\).
	If \(J\subseteq I\) and \(c_I\ne0\), equation \eqref{eq:coefficient-transport} gives \(c_I=\bpartial_{\by}^{I\setminus J}c_J\), and hence \(c_J\ne0\).
	Thus \(\cN(p)\) is ordered.
	In particular, \(\sO(\bv)\) is nonempty.

	The union in \eqref{eq:largest-ordered-family} is also an ordered negative support family, so it is the largest one.
	The chosen series is supported on \(\cN(p)\subseteq\cN_{\bv}^{\ord}\).
	Theorem~5.5 of \cite{OS25} therefore gives \(p\in V_{\bv,\cN_{\bv}^{\ord}}^{\amb}\).
	This proves that the right-hand side of the first equality in \eqref{eq:largest-ordered-family-full-inverse-system} is contained in the left-hand side.
	Conversely, Theorem~5.5 of \cite{OS25} identifies \(V_{\bv,\cN_{\bv}^{\ord}}^{\amb}\) with the full coefficient space at \(\bv\) for \(\cN_{\bv}^{\ord}\).
	Since \(\cN_{\bv}^{\ord}\subseteq\cN_{\bv}\), this space is contained in the full coefficient space at \(\bv\) for \(\cN_{\bv}\), which Theorem~4.9 of \cite{OS25} and the definition of \(\cE_{\bv}\) identify with \(\cE_{\bv}\).
	Because the ambient coefficient space \(V_{\bv,\cN_{\bv}^{\ord}}^{\amb}\) equals the full coefficient space \(\cE_{\bv}\) for \(\cN_{\bv}\) and \(\langle\cdot,\cdot\rangle_{\bs}\) is perfect in each degree, the second equality in \eqref{eq:largest-ordered-family-full-inverse-system} follows.
\end{proof}

Define \(\cI_{\bv}^{L}\), \(J_{\bv}^{L}\), and \(\delta_{\bv}^{L}\) by \(\cI_{\bv}^{L}=\sum_{\cN\in\sO(\bv)}V_{\bv,\cN}^{\intr}\), \(J_{\bv}^{L}=\bigcap_{\cN\in\sO(\bv)}J_{\bv,\cN}^{\intr}\), and \(\delta_{\bv}^{L}=\dim_{\C}\bigl(\cE_{\bv}/\cI_{\bv}^{L}\bigr)\).
The space \(\cI_{\bv}^{L}\) and the quotient \(\cE_{\bv}/\cI_{\bv}^{L}\) are finite-dimensional because \(\cE_{\bv}\) is finite-dimensional.

\begin{proposition}
	\label{prop:all-ordered-local-quotient}
	For every \(\bv\in\sE_{\bbeta,\bw}\), we have
	\begin{equation}
		P_B(\bv)
		\subseteq
		\cE_{\bv}^{\perp}
		=
		J_{\bv,\cN_{\bv}^{\ord}}^{\amb}
		\subseteq
		J_{\bv}^{L},
		\qquad
		\cI_{\bv}^{L}
		=
		(J_{\bv}^{L})^\perp,
		\label{eq:all-ordered-ideal-inclusion}
	\end{equation}
	and \(\langle\cdot,\cdot\rangle_{\bs}\) gives
	\begin{equation}
		\delta_{\bv}^{L}
		=
		\dim_{\C}
		\left(
		\frac{J_{\bv}^{L}}
		{\cE_{\bv}^{\perp}}
		\right).
		\label{eq:all-ordered-local-ideal-codimension}
	\end{equation}
\end{proposition}

\begin{proof}
	For every \(\cN\in\sO(\bv)\), a series supported on \(\cN\) is also supported on \(\cN_{\bv}^{\ord}\), so Lemma~\ref{lem:largest-ordered-family} gives \(V_{\bv,\cN}^{\amb}\subseteq\cE_{\bv}\).
	Fact~\ref{fact:defect-coefficient-duality} then gives \(V_{\bv,\cN}^{\intr}\subseteq\cE_{\bv}\).
	Taking sums of coefficient spaces and annihilators gives \eqref{eq:all-ordered-ideal-inclusion}.
	Since the ideals are homogeneous, the pairing \(\langle\cdot,\cdot\rangle_{\bs}\) identifies the graded dual of the quotient in \eqref{eq:all-ordered-local-ideal-codimension} with \(\cE_{\bv}/\cI_{\bv}^{L}\).
	Both spaces are finite-dimensional, so their dimensions are equal.
\end{proof}

\begin{remark}
	If the local components of the indicial ideal and the fake indicial ideal agree at \(\bv\), then the inclusions in \eqref{eq:actual-fake-local-inclusions} are equalities and the denominator in \eqref{eq:all-ordered-local-ideal-codimension} can be replaced by \(P_B(\bv)\).
\end{remark}

For \(\bv,\bv'\in\sE_{\bbeta,\bw}\), define \(\bv\sim_{\bw}\bv'\) to mean that \(\bv'-\bv\in L\) and \(\bw\cdot(\bv'-\bv)=0\).
Let \(\sP_{\bbeta,\bw}\) be the set of equivalence classes for this relation.
Define \(P\prec_{\bw}P'\) to mean that \(\bw\cdot(\bv'-\bv)>0\) for \(\bv\in P\) and \(\bv'\in P'\).
This relation totally orders the classes contained in one \(L\)-coset.
Define \(\sP_{\bbeta,\bw}^{\min}\) to be the set of the least classes in the \(L\)-cosets.
For \(P\in\sP_{\bbeta,\bw}\), define \(\cE_P\) and \(\cI_P^{L}\) by
\begin{equation}
	\cE_P
	=
	\bigoplus_{\bv\in P}\cE_{\bv},
	\qquad
	\cI_P^{L}
	=
	\bigoplus_{\bv\in P}\cI_{\bv}^{L}.
	\label{eq:formal-weight-class-local-spaces}
\end{equation}

Let \(\cV_{\bw}\) denote the canonical formal solution space in the direction \(\bw\), that is, the span of the formal series solutions \(\Xi_{\fs}\) in Fact~\ref{fact:canonical-series-construction}.
Define \(\cV_{\bw}^{L}\) to be the span of the series
\[
	\left.
	q(\bpartial_{\bs})\Psi_{\cN}(\bx,\bs)
	\right|_{\bs=\bzero}
\]
as \(\bv\in\sE_{\bbeta,\bw}\), \(\cN\in\sO(\bv)\), and \(q(\bpartial_{\bs})\in\bigl(\Phi_0(P_{\cN}^0(\bt))\bigr)^{\perp}\) vary.
These series belong to \(\cV_{\bw}\) by \cite[Theorem~3.2]{OS25} and Proposition~\ref{prop:canonical-nilsson-identification}.
They need not equal individual canonical series \(\Xi_{\fs}\) in Fact~\ref{fact:canonical-series-construction}.

\begin{proposition}
	\label{prop:canonical-nilsson-identification}
	Under the standing homogeneity assumption on \(A\) and the standing genericity assumption on \(\bw\), there exist a positive weight vector \(\bw'\) and a strongly convex open rational polyhedral cone \(\cQ\subseteq\R_{>0}^n\) such that
	\begin{align*}
		\bw' & \in\cQ,
		\qquad
		\bone\in\overline{\cQ},
		\qquad
		\operatorname{in}_{\bw'}(I_A)=\inw(I_A), \\
		\operatorname{in}_{(-\bw',\bw')}\bigl(H_A(\bbeta)\bigr)
		     & =
		\operatorname{in}_{(-\bw,\bw)}\bigl(H_A(\bbeta)\bigr).
	\end{align*}
	For this choice of \(\bw'\), the canonical formal solution space \(\cV_{\bw}\) is the \(\C\)-span of the basic Nilsson solutions of \(H_A(\bbeta)\) in the direction \(\bw'\).
	For every starting monomial, the corresponding canonical series is the unique basic Nilsson solution whose coefficient at that starting monomial is one and in which no other starting monomial occurs.
	An arbitrary basic Nilsson solution is, in general, a unique finite linear combination of these canonical series and need not equal any one of them.
\end{proposition}

\begin{proof}
	Put \(\Lambda=\Z A\), choose a \(\Z\)-basis of \(\Lambda\), and let \(T\in\Z^{d\times d}\) have these basis vectors as its columns.
	Put \(A'=T^{-1}A\) and \(\bbeta'=T^{-1}\bbeta\).
	Then \(A'\) is integral, \(\Z A'=\Z^d\), and \(\ker_{\Z}(A')=\ker_{\Z}(A)=L\).
	The identity \(A\tp{(x_1\partial_1,\ldots,x_n\partial_n)}-\bbeta = T\left( A'\tp{(x_1\partial_1,\ldots,x_n\partial_n)}-\bbeta' \right)\) and the equality of the integer kernels give \(H_A(\bbeta)=H_{A'}(\bbeta')\) as left ideals in the same Weyl algebra.
	This replacement changes neither the canonical series nor the basic Nilsson solutions, so we may assume \(\Z A=\Z^d\), as in \cite{DMM12}.

	Choose \(\bh\) such that \(\bh\ba_j=1\) for \(j=1,\ldots,n\).
	Then \(\bone=\bh A\) belongs to \(\operatorname{rowspan}_{\R}(A)\), and \(|\bu| = \bone\cdot\bu = \bh(A\bu) = 0\) for \(\bu\in L\).
	Let \(\cQ_0\) be the common full-dimensional open rational polyhedral Gr\"obner cone containing \(\bw\) on which the toric and Weyl algebra initial ideals are constant.
	If \(\br\in\operatorname{rowspan}_{\R}(A)\), then adding \(\br\) to a weight adds the same constant to the weights of all terms of an \(A\)-homogeneous polynomial or differential operator.
	Hence \(\cQ_0\) is invariant under addition by \(\operatorname{rowspan}_{\R}(A)\); the same calculation appears in \cite[Remark~2.5]{DMM12}.
	Choose \(\lambda>0\) sufficiently large and put \(\bw'=\bw+\lambda\bone\).
	Then \(\bw'>0\), the displayed initial ideals agree, and \(\bw'\cdot\bu=\bw\cdot\bu\) for \(\bu\in L\).
	Put \(\cQ=\cQ_0\cap\R_{>0}^n\).
	The cone \(\cQ\) is open, rational polyhedral, and strongly convex, and it contains \(\bw'\).
	For every sufficiently large positive real number \(\mu\), the vector \(\mu^{-1}(\bw+\mu\bone)=\bone+\mu^{-1}\bw\) belongs to \(\cQ\), and hence \(\bone\in\overline{\cQ}\).
	Thus \(\bw'\) is a perturbation of \(\bone\) in the sense of \cite[Definition~3.3]{DMM12}.

	Fix the monomial order refining the \(\bw\)-weight as in Fact~\ref{fact:canonical-series-construction}.
	The equality \(\bw'\cdot\bu=\bw\cdot\bu\) for \(\bu\in L\) shows that the same order refines the \(\bw'\)-weight on every \(L\)-coset.
	By Fact~\ref{fact:canonical-series-construction}, the solutions of \(\operatorname{in}_{(-\bw,\bw)}(H_A(\bbeta))\) are finite sums of terms \(\bx^{\bv}q(\log\bx)\), where \(\bv\) is an exponent and \(q\) belongs to the local inverse system of the indicial ideal at \(\bv\).
	Fact~\ref{fact:canonical-series-construction} also shows that the set of starting monomials is finite and that each starting monomial \(\fs\) determines a unique canonical series \(\Xi_{\fs}\) whose coefficient at \(\fs\) is one and in which no other starting monomial occurs.
	As \(\fs\) ranges over all starting monomials, the series \(\Xi_{\fs}\) form a basis of \(\cV_{\bw}\).

	The support condition for canonical series in \cite[Section~2.5]{SST00} shows that the lattice shifts of \(\Xi_{\fs}\) lie in \(\cQ_0^*\cap L\), and hence in \(\cQ^*\cap L\).
	The logarithmic degrees in each \(\Xi_{\fs}\) have a common finite bound, and the coefficient of the zero shift is nonzero.
	Therefore every \(\Xi_{\fs}\) satisfies the three conditions in \cite[Definition~2.6]{DMM12} and is a basic Nilsson solution in the direction \(\bw'\).
	This proves
	\[
		\cV_{\bw}
		\subseteq
		\Span
		\left\{
		\text{basic Nilsson solutions in the direction \(\bw'\)}
		\right\}.
	\]

	Conversely, let \(\phi\) be a basic Nilsson solution in the direction \(\bw'\).
	By comparing the least \(\bw'\)-weight terms in the equations \(H_A(\bbeta)\phi=0\), we see that the initial series of \(\phi\) is annihilated by \(\operatorname{in}_{(-\bw,\bw)}(H_A(\bbeta))\).
	Write the initial series of \(\phi\) as \(\bx^{\bv}p(\log\bx)\).
	By \cite[Lemma~2.10]{DMM12}, the vector \(\bv\) is an exponent of \(H_A(\bbeta)\) with respect to \(\bw'\).
	Thus every starting monomial occurring in \(\phi\) belongs to the same finite set that indexes the series \(\Xi_{\fs}\).
	For each starting monomial \(\fs\), let \(a_{\fs}(\phi)\) be its coefficient in \(\phi\), and put \(\phi_0 = \phi - \sum_{\fs} a_{\fs}(\phi)\Xi_{\fs}\).
	The defining coefficient conditions for the series \(\Xi_{\fs}\) show that no starting monomial occurs in \(\phi_0\).
	Suppose that \(\phi_0\ne0\).
	Since \(\inw(I_A)\) is monomial, the ideal \(\operatorname{in}_{(-\bw,\bw)}(H_A(\bbeta))\) annihilates separately the components of the initial series of \(\phi_0\) at the distinct exponents.
	Hence, with respect to the monomial order refining the \(\bw\)-weight, the initial monomial of \(\phi_0\) is a starting monomial.
	This contradicts the construction of \(\phi_0\).
	Hence \(\phi_0=0\), and \(\phi\) is a finite linear combination of the series \(\Xi_{\fs}\).
	The same coefficient conditions make this linear combination unique.

	If the initial solutions contain logarithms, the starting monomials have the form \(\bx^{\bv}(\log\bx)^{\bk}\), and distinct multi-indices \(\bk\) give distinct starting monomials and hence index distinct series \(\Xi_{\fs}\) in the basis of \(\cV_{\bw}\).
	Thus the correspondence retains both the exponent \(\bv\) and the multi-index \(\bk\) of each starting monomial \(\bx^{\bv}(\log\bx)^{\bk}\).
	By \cite[Lemma~2.8]{DMM12}, independent initial series lift to independent Nilsson solutions, and any independent family of Nilsson solutions can be recombined so that its initial series are independent.
	A basic Nilsson solution equals one of the series \(\Xi_{\fs}\) only when exactly one of its coefficients at the starting monomials is nonzero and that coefficient is one.
\end{proof}

For exponents \(\balpha\) and \(\balpha'\) in one \(L\)-coset, the number \(\bw\cdot(\balpha'-\balpha)\) is the relative \(\bw\)-weight of \(\balpha'\) with respect to \(\balpha\).
This number depends only on the two exponents, so it orders each \(L\)-coset without reference to a base point.

\begin{lemma}
	\label{lem:least-term}
	Let \(0\ne\phi\in\cV_{\bw}\) be supported on one \(L\)-coset.
	The support of \(\phi\) is \(\bw\)-well-ordered: every nonempty subset of the support contains an element of least relative \(\bw\)-weight.
	Define \(\phi_{\min}\) to be the sum of the terms of least relative \(\bw\)-weight.
	Then
	\begin{equation}
		\operatorname{in}_{(-\bw,\bw)}\bigl(H_A(\bbeta)\bigr)\phi_{\min}=0.
		\label{eq:least-term-initial-system}
	\end{equation}
	Every exponent with a nonzero coefficient in \(\phi_{\min}\) is a fake exponent and belongs to \(\sE_{\bbeta,\bw}\).
\end{lemma}

\begin{proof}
	By \cite[Definition~2.6]{DMM12}, the shifts in the support of a basic Nilsson solution are lattice points in the dual of a strongly convex open cone containing \(\bw\).
	A set of such lattice points is finite if the \(\bw\)-weight is bounded on that set.
	An element of \(\cV_{\bw}\) is a finite sum of basic Nilsson solutions.
	Thus the support is \(\bw\)-well-ordered, and \(\phi_{\min}\) is a nonzero finite sum.
	The assertion about \(\phi_{\min}\) also follows from \cite[Proposition~2.5.2]{SST00}.
	For an operator \(P\in H_A(\bbeta)\), the terms of least \(\bw\)-weight in \(P\phi\) are obtained by applying \(\operatorname{in}_{(-\bw,\bw)}(P)\) to \(\phi_{\min}\).
	Since \(P\phi=0\), equation \eqref{eq:least-term-initial-system} follows.
	This equation is the assertion about the initial series in \cite[Theorem~2.5.5]{SST00}.
	Write \(\phi_{\min}\) as a finite sum of terms \(\bx^{\balpha}p_{\balpha}(\log\bx)\) with distinct exponents and nonzero polynomials \(p_{\balpha}\).
	For every polynomial \(f\) in the fake indicial ideal, equation \eqref{eq:least-term-initial-system} and comparison of the distinct exponents give \(f(\balpha+\bpartial_{\by})p_{\balpha}=0\).
	Comparison of the terms of the highest logarithmic degree gives \(f(\balpha)=0\), so every such \(\balpha\) is a fake exponent.
	This argument treats the whole sum \(\phi_{\min}\), including interactions among terms of the same weight, and extends the conclusion of \cite[Lemma~2.10]{DMM12}, which is stated for a basic Nilsson solution.
	The initial ideal \(\inw(I_A)\) is monomial by genericity, so the components of \(\phi_{\min}\) at the distinct fake exponents are separately annihilated by the initial system.
	Decompose \(\phi_{\min}\) into these components, express each component in the basis of solutions of the initial system given in Fact~\ref{fact:canonical-series-construction}, and replace each basis element by its uniquely determined canonical series.
	Every nonzero component then occurs in the canonical formal solution space with its displayed exponent.
	Thus \(\cE_{\balpha}\ne0\), and hence \(\balpha\in\sE_{\bbeta,\bw}\).
\end{proof}

Decompose \(\cV_{\bw}\) by the \(L\)-cosets of its exponents, and filter each summand by the relative \(\bw\)-weights of the initial exponents of that summand.
These filtrations induce filtrations on the corresponding summands of \(\cV_{\bw}^{L}\).
Write \(\gr\) for the direct sum of the resulting associated graded spaces.
For \(P\in\sP_{\bbeta,\bw}\), define \(\cW_P\) to be the image in \(\cE_P\) of the component indexed by \(P\) in \(\gr\cV_{\bw}^{L}\).

\begin{theorem}
	\label{thm:formal-solution-codimension-bounds}
	The spaces \(\cW_P\) satisfy
	\begin{equation}
		\cI_P^{L}
		\subseteq
		\cW_P,
		\qquad
		\cW_P=\cI_P^{L}
		\quad
		(P\in\sP_{\bbeta,\bw}^{\min}).
		\label{eq:formal-filtration-local-spaces}
	\end{equation}
	The associated graded spaces satisfy
	\begin{equation}
		\gr\cV_{\bw}\simeq\bigoplus_{P\in\sP_{\bbeta,\bw}}\cE_P, \quad \gr\cV_{\bw}^{L}\simeq\bigoplus_{P\in\sP_{\bbeta,\bw}}\cW_P, \quad \gr\bigl(\cV_{\bw}/\cV_{\bw}^{L}\bigr)\simeq\bigoplus_{P\in\sP_{\bbeta,\bw}}\cE_P/\cW_P.
		\label{eq:formal-filtration-graded}
	\end{equation}
	The same filtration gives
	\begin{equation}
		\dim_{\C}\cV_{\bw}=\sum_{P\in\sP_{\bbeta,\bw}}\dim_{\C}\cE_P, \qquad \dim_{\C}\cV_{\bw}^{L}=\sum_{P\in\sP_{\bbeta,\bw}}\dim_{\C}\cW_P.
		\label{eq:formal-filtration-dimensions}
	\end{equation}
	In particular,
	\begin{equation}
		\dim_{\C}
		\left(
		\cV_{\bw}/\cV_{\bw}^{L}
		\right)
		=
		\sum_{P\in\sP_{\bbeta,\bw}}
		\dim_{\C}
		\left(
		\cE_P/\cW_P
		\right).
		\label{eq:formal-filtration-exact-codimension}
	\end{equation}
	Consequently,
	\begin{equation}
		\sum_{P\in\sP_{\bbeta,\bw}^{\min}}
		\sum_{\bv\in P}\delta_{\bv}^{L}
		\leq
		\dim_{\C}
		\left(
		\cV_{\bw}/\cV_{\bw}^{L}
		\right)
		\leq
		\sum_{\bv\in\sE_{\bbeta,\bw}}
		\delta_{\bv}^{L}.
		\label{eq:formal-solution-codimension-bounds}
	\end{equation}
	Equality holds in the upper bound if and only if \(\cW_P=\cI_P^{L}\) for every \(P\in\sP_{\bbeta,\bw}\), and equality holds in the lower bound if and only if \(\cW_P=\cE_P\) for every \(P\notin\sP_{\bbeta,\bw}^{\min}\).
\end{theorem}

\begin{proof}
	Partition the members of \(\sE_{\bbeta,\bw}\) into their \(L\)-cosets and then into the classes in \(\sP_{\bbeta,\bw}\).
	Series supported on distinct cosets have no monomial in common, so both \(\cV_{\bw}\) and \(\cV_{\bw}^{L}\) are direct sums over these cosets.
	On these summands we use the direct sum of the filtrations defined below.
	Fix one coset \(C\), and list its classes as \(P_1\prec_{\bw}\cdots\prec_{\bw}P_m\).
	Choose an exponent \(\bv_C\) in \(C\), and define \(a_k\) by \(a_k=\bw\cdot(\bv-\bv_C)\) for \(\bv\in P_k\).
	The number \(a_k\) is well-defined, and \(a_1<\cdots<a_m\).
	Let \(F_{C,k}\) be the subspace of \(\cV_{\bw}\) consisting of series supported on \(C\) whose nonzero terms have relative weight at least \(a_k\), and put \(F_{C,m+1}=0\).

	Each series \(\Xi_{\fs}\) has only shifts of positive \(\bw\)-weight away from the exponent of its starting monomial \(\fs\); see \cite[Section~2.5]{SST00}.
	Fact~\ref{fact:canonical-series-construction} describes the logarithmic coefficients at the exponents in \(P_k\) and shows that each prescribed starting monomial determines a unique canonical series.
	These results identify the quotient \(F_{C,k}/F_{C,k+1}\) by means of the tuple of coefficients at the exponents in \(P_k\):
	\begin{equation}
		F_{C,k}/F_{C,k+1}
		\simeq
		\bigoplus_{\bv\in P_k}\cE_{\bv}
		=
		\cE_{P_k}.
		\label{eq:canonical-series-filtration-quotient}
	\end{equation}
	Every tuple on the right lifts to the sum of the canonical series with the prescribed initial terms.
	Conversely, suppose that the coefficients of a series in \(F_{C,k}\) vanish at every exponent in \(P_k\).
	If this series is nonzero, Lemma~\ref{lem:least-term} shows that every exponent bearing a term of least relative weight in the series belongs to \(\sE_{\bbeta,\bw}\).
	For \(k<m\), every nonzero series in the kernel therefore has least relative weight at least \(a_{k+1}\); hence the kernel of the coefficient map is \(F_{C,k+1}\).
	For \(k=m\), the same argument shows that the kernel is zero.

	Define \(W_C\) by \(W_C=\cV_{\bw}^{L}\cap F_{C,1}\), and define \(\cW_{P_k}\) to be the image of \(F_{C,k}\cap W_C\) in the quotient \eqref{eq:canonical-series-filtration-quotient}.
	Summing over the cosets gives the three isomorphisms in \eqref{eq:formal-filtration-graded}, and the filtrations of \(F_{C,1}\) and \(W_C\) give \(\dim_{\C}F_{C,1} = \sum_k\dim_{\C}\cE_{P_k}\) and \(\dim_{\C}W_C = \sum_k\dim_{\C}\cW_{P_k}\), whence the two identities in \eqref{eq:formal-filtration-dimensions} and, on subtracting, \eqref{eq:formal-filtration-exact-codimension}.

	For every \(\bv\in P_k\), the outputs of intrinsic perturbation with \(\bv\) as the base of the construction belong to \(F_{C,k}\cap W_C\).
	Their coefficients at \(\bx^{\bv}\), as the ordered negative support family varies, span \(\cI_{\bv}^{L}\).
	Their coefficient tuples indexed by \(P_k\) have no nonzero component at any other exponent.
	Therefore \(\cI_{P_k}^{L}\subseteq\cW_{P_k}\).
	For \(k=1\), an output of intrinsic perturbation with an exponent in \(P_j\), \(j>1\), as the base of the construction contains no term indexed by \(P_1\).
	An output with \(\bv\in P_1\) as the base of the construction contains no term indexed by \(P_1\) other than \(\bv\), because an occurring nonzero shift has positive \(\bw\)-weight.
	Hence the coefficient tuple of every element of \(W_C\) lies in \(\cI_{P_1}^{L}\) and \(\cW_{P_1}=\cI_{P_1}^{L}\), which proves \eqref{eq:formal-filtration-local-spaces}.
	The inclusion there gives \(\dim_{\C}(\cE_P/\cW_P) \leq \dim_{\C}(\cE_P/\cI_P^{L}) = \sum_{\bv\in P}\delta_{\bv}^{L}\) for every \(P\).
	Together with the equality for the least classes, this bound proves the two inequalities in \eqref{eq:formal-solution-codimension-bounds}.
	The descriptions of equality follow term by term from the same inclusions.
\end{proof}

We next describe the spaces \(\cW_P\) in terms of linear maps between finite-dimensional spaces.
Put \(\cP_{\bv,\cN}=\Phi_0\bigl(P_{\cN}^0(\bt)\bigr) \subseteq S_0\).
Thus \(\cP_{\bv,\cN}\) is the polynomial ideal whose completion is \(\Pi_{\cN}\) before the colon by \(m_{\bv,\cN}\) is taken.
The intrinsic perturbation construction in \cite[Theorem~3.2]{OS25} gives a linear map
\begin{equation}
	\sF_{\bv,\cN}:
	\cP_{\bv,\cN}^{\perp}
	\longrightarrow
	\cV_{\bw},
	\qquad
	q(\bpartial_{\bs})
	\longmapsto
	\left.
	q(\bpartial_{\bs})
	\Psi_{\cN}(\bx,\bs)
	\right|_{\bs=\bzero}.
	\label{eq:full-intrinsic-series-map}
\end{equation}
The image of \(\sF_{\bv,\cN}\) is finite-dimensional because it is contained in \(\cV_{\bw}\).
By the definition of \(\cV_{\bw}^{L}\), we have
\[
	\cV_{\bw}^{L}
	=
	\sum_{\bv\in\sE_{\bbeta,\bw}}
	\sum_{\cN\in\sO(\bv)}
	\im\sF_{\bv,\cN}.
\]
For a formal series \(\phi\), write \(\operatorname{coeff}_{\bv'}(\phi)\) for the polynomial in \(\log\bx\) that multiplies \(\bx^{\bv'}\).
If \(\bv'=\bv+\bu\) and \(I_{\bu}\in\cN\), the definition of the perturbation series gives
\begin{equation}
	\operatorname{coeff}_{\bv'}
	\bigl(\sF_{\bv,\cN}(q)\bigr)
	=
	\left.
	q(\bpartial_{\bs})
	\left(
	m_{\bv,\cN}(\bs)
	a_{\bu}(\bs)
	\exp\bigl((\log\bx)B\bs\bigr)
	\right)
	\right|_{\bs=\bzero}.
	\label{eq:full-intrinsic-transition-coefficient}
\end{equation}
The coefficient \(\operatorname{coeff}_{\bv'}\bigl(\sF_{\bv,\cN}(q)\bigr)\) is zero when \(\bv'=\bv+\bu\) with \(I_{\bu}\notin\cN\).

Fix an \(L\)-coset \(C\) that meets \(\sE_{\bbeta,\bw}\), write \(m(C)\) for the number of classes contained in \(C\), put \(m=m(C)\), and write these classes as \(P_1\prec_{\bw}\cdots\prec_{\bw}P_m\).
Choose \(\bv_C\in C\), and put \(a_k=\bw\cdot(\bv-\bv_C)\) for \(\bv\in P_k\).
Let \(F_{C,k}\) be the filtration of the subspace of \(\cV_{\bw}\) consisting of series supported on \(C\) that is constructed in the proof of Theorem~\ref{thm:formal-solution-codimension-bounds}.
That proof gives surjective coefficient maps
\begin{equation}
	\pi_{C,k}:F_{C,k}\longrightarrow
	\cE_{P_k},
	\qquad
	\ker\pi_{C,k}=F_{C,k+1}.
	\label{eq:formal-transition-quotient-map}
\end{equation}
Choose a linear section \(\sigma_{C,k}:\cE_{P_k}\to F_{C,k}\) of \(\pi_{C,k}\), and put
\begin{equation}
	\Lambda_C:
	\bigoplus_{k=1}^m\cE_{P_k}
	\longrightarrow F_{C,1},
	\qquad
	(p_1,\ldots,p_m)
	\longmapsto
	\sum_{k=1}^m\sigma_{C,k}(p_k).
	\label{eq:formal-transition-lifting-map}
\end{equation}
For \(1\leq k\leq m\), put
\begin{equation}
	\cD_{C,k}
	=
	\bigoplus_{\bv\in P_k}
	\bigoplus_{\cN\in\sO(\bv)}
	\im\sF_{\bv,\cN},
	\qquad
	\cD_C
	=
	\bigoplus_{k=1}^m\cD_{C,k},
	\label{eq:formal-transition-source}
\end{equation}
where the sums in \(\cD_{C,k}\) are external direct sums.
Let \(\Sigma_C:\cD_C\to F_{C,1}\) be the map that sends a tuple of components to the sum of the series that the components represent.
By \eqref{eq:full-intrinsic-transition-coefficient}, the blocks of \(\Lambda_C^{-1}\Sigma_C\) are computed from the coefficients at the finitely many exponents in \(C\) by successive subtraction of the images of the sections \(\sigma_{C,k}\).

Let \(F_1\supseteq F_2\supseteq\cdots\supseteq F_{m+1}=0\) be a finite filtration of a \(\C\)-vector space.
For \(1\leq k\leq m\), define \(\overline F_k\) by \(\overline F_k=F_k/F_{k+1}\), and let \(\pi_k:F_k\to\overline F_k\) be the quotient map.
Choose a linear section \(\sigma_k:\overline F_k\to F_k\) of \(\pi_k\) for every \(k\).
Define \(\Lambda:\bigoplus_{k=1}^m\overline F_k\to F_1\) by \(\Lambda(\overline f_1,\ldots,\overline f_m)=\sum_k\sigma_k(\overline f_k)\).
Let \(\mathsf X_1,\ldots,\mathsf X_m\) be \(\C\)-vector spaces, and define \(\mathsf X\) by \(\mathsf X=\bigoplus_{i=1}^m\mathsf X_i\).
Let \(\Sigma:\mathsf X\to F_1\) be a linear map with \(\Sigma(\mathsf X_i)\subseteq F_i\) for every \(i\).
Define \(\mathsf T\) by \(\mathsf T=\Lambda^{-1}\Sigma\) and \(\mathsf T_{ji}:\mathsf X_i\to\overline F_j\) by \(\mathsf T_{ji}=\pr_j\mathsf T|_{\mathsf X_i}\), where \(\pr_j\) denotes the projection onto \(\overline F_j\).
For \(k\geq2\), define \(\mathsf T_{<k}\) by \(\mathsf T_{<k}=(\pr_1,\ldots,\pr_{k-1})\mathsf T|_{\bigoplus_{i<k}\mathsf X_i}\) and \(\mathsf T_{k,<k}\) by \(\mathsf T_{k,<k}=\pr_k\mathsf T|_{\bigoplus_{i<k}\mathsf X_i}\).
For \(1\leq k\leq m\), define \(W_k\) to be the image of \(F_k\cap\im\Sigma\) under \(\pi_k\), so that \(\pi_k\) identifies \((F_k\cap\im\Sigma)/(F_{k+1}\cap\im\Sigma)\) with \(W_k\).

The following lemma describes the induced filtration on the image of a filtered map.
\begin{lemma}
	\label{lem:filtered-associated-graded-image}
	The map \(\Lambda\) is an isomorphism.
	Moreover, \(\mathsf T_{ji}=0\) for \(j<i\), and \(W_1=\im\mathsf T_{11}\) and \(W_k=\im\mathsf T_{kk}+\mathsf T_{k,<k}(\ker\mathsf T_{<k})\) for \(k\geq2\).
\end{lemma}

\begin{proof}
	Take \(\widetilde f\in F_1\), define \(\widetilde f_1\) by \(\widetilde f_1=\widetilde f\), and for \(1\leq k\leq m\) define \(\overline f_k\) by \(\overline f_k=\pi_k(\widetilde f_k)\) and \(\widetilde f_{k+1}\) by \(\widetilde f_{k+1}=\widetilde f_k-\sigma_k(\overline f_k)\).
	Then \(\widetilde f_{k+1}\in F_{k+1}\) and \(\widetilde f_{m+1}=0\), so \(\widetilde f=\sum_k\sigma_k(\overline f_k)\) and \(\Lambda\) is surjective.
	If \(\Lambda(\overline f_1,\ldots,\overline f_m)=0\), reduction modulo \(F_2\) gives \(\overline f_1=0\), and after \(\overline f_1,\ldots,\overline f_{k-1}\) have been shown to be zero, reduction modulo \(F_{k+1}\) gives \(\overline f_k=0\); thus \(\Lambda\) is injective.
	Since \(\sigma_j(\overline F_j)\subseteq F_j\), the map \(\Lambda\) sends \(\bigoplus_{j\geq k}\overline F_j\) into \(F_k\), and successive reduction modulo \(F_2,\ldots,F_k\) gives the reverse inclusion, so \(\Lambda^{-1}(F_k)=\bigoplus_{j\geq k}\overline F_j\).
	The inclusion \(\Sigma(\mathsf X_i)\subseteq F_i\) now gives \(\mathsf T(\mathsf X_i)\subseteq\bigoplus_{j\geq i}\overline F_j\), so \(\mathsf T_{ji}=0\) for \(j<i\).
	For an element of \(\bigoplus_{j\geq k}\overline F_j\), the map \(\pi_k\Lambda\) is the projection onto \(\overline F_k\), so \(W_k=\pr_k(\im\mathsf T\cap\bigoplus_{j\geq k}\overline F_j)\).
	Choose \(\xi_i\in\mathsf X_i\) for \(1\leq i\leq m\).
	The first \(k-1\) coordinates of \(\mathsf T(\xi_1,\ldots,\xi_m)\) are \(\mathsf T_{<k}(\xi_1,\ldots,\xi_{k-1})\), because the components \(\xi_k,\ldots,\xi_m\) have zero coordinates in \(\overline F_1,\ldots,\overline F_{k-1}\).
	Hence \(\mathsf T(\xi_1,\ldots,\xi_m)\in\bigoplus_{j\geq k}\overline F_j\) if and only if \((\xi_1,\ldots,\xi_{k-1})\in\ker\mathsf T_{<k}\), and under this condition the \(k\)-th coordinate of \(\mathsf T(\xi_1,\ldots,\xi_m)\) is \(\mathsf T_{k,<k}(\xi_1,\ldots,\xi_{k-1})+\mathsf T_{kk}(\xi_k)\), because the components \(\xi_{k+1},\ldots,\xi_m\) have zero coordinate in \(\overline F_k\).
	This proves the formula for \(W_k\) when \(k\geq2\), and for \(k=1\) lower triangularity gives \(\pr_1\mathsf T(\xi_1,\ldots,\xi_m)=\mathsf T_{11}(\xi_1)\).
\end{proof}

Each component of \(\cD_{C,i}\) represents a series obtained by intrinsic perturbation with an exponent in \(P_i\) as the base of the construction.
Every term of this series has relative weight at least \(a_i\), so the series belongs to \(F_{C,i}\) and \(\Sigma_C(\cD_{C,i})\subseteq F_{C,i}\).
The inclusions \(\Sigma_C(\cD_{C,i})\subseteq F_{C,i}\) allow us to apply Lemma~\ref{lem:filtered-associated-graded-image} with \(F_k=F_{C,k}\), \(\overline F_k=\cE_{P_k}\), \(\mathsf X_i=\cD_{C,i}\), and \(\Sigma=\Sigma_C\).
In this application, the map \(\Lambda\) is \(\Lambda_C\), so \(\Lambda_C\) is an isomorphism.
Define
\begin{equation}
	\mathsf T_C
	=
	\Lambda_C^{-1}\Sigma_C:
	\cD_C
	\longrightarrow
	\bigoplus_{k=1}^m\cE_{P_k}.
	\label{eq:formal-transition-map}
\end{equation}
For the application of Lemma~\ref{lem:filtered-associated-graded-image} above, the map \(\mathsf T\) is \(\mathsf T_C\).
Let \(\mathsf T_{C,ji}:\cD_{C,i}\to\cE_{P_j}\) be the \((j,i)\)-block of \(\mathsf T_C\).
For \(k\geq2\), define
\begin{align}
	\mathsf T_{C,<k}
	&=
	(\pr_1,\ldots,\pr_{k-1})
	\mathsf T_C
	\big|_{\bigoplus_{i<k}\cD_{C,i}},
	\label{eq:formal-transition-preceding-map}\\
	\mathsf T_{C,k,<k}
	&=
	\pr_k
	\mathsf T_C
	\big|_{\bigoplus_{i<k}\cD_{C,i}}.
	\label{eq:formal-transition-current-map}
\end{align}

The following theorem gives the cokernel of \(\mathsf T_C\) and the images in the associated graded spaces.

\begin{theorem}
	\label{thm:formal-solution-exact-cokernel}
	There is an exact sequence
	\begin{equation}
		\cD_C
		\xrightarrow{\mathsf T_C}
		\bigoplus_{k=1}^m\cE_{P_k}
		\longrightarrow
		\frac{F_{C,1}}
		{F_{C,1}\cap\cV_{\bw}^{L}}
		\longrightarrow0.
		\label{eq:formal-transition-cokernel}
	\end{equation}
	The maps \(\mathsf T_{C,ji}\) satisfy
	\begin{equation}
		\mathsf T_{C,ji}=0\quad(j<i),
		\qquad
		\im\mathsf T_{C,ii}
		=
		\cI_{P_i}^{L}.
		\label{eq:formal-transition-diagonal}
	\end{equation}
	The images in the associated graded spaces satisfy
	\begin{equation}
		\cW_{P_k}
		=
		\cI_{P_k}^{L}
		+
		\mathsf T_{C,k,<k}
		\bigl(\ker\mathsf T_{C,<k}\bigr)
		\qquad(k\geq2),
		\label{eq:formal-transition-associated-graded-image}
	\end{equation}
	and \(\cW_{P_1}=\cI_{P_1}^{L}\).
\end{theorem}

\begin{proof}
	By the definition of \(\cV_{\bw}^{L}\) and the decomposition by \(L\)-cosets, we have \(\im\Sigma_C=F_{C,1}\cap\cV_{\bw}^{L}\).
	Since \(\Lambda_C\) is an isomorphism, the cokernel of \(\mathsf T_C\) is isomorphic to \(F_{C,1}/(F_{C,1}\cap\cV_{\bw}^{L})\), which gives \eqref{eq:formal-transition-cokernel}.
	Lemma~\ref{lem:filtered-associated-graded-image} gives \(\mathsf T_{C,ji}=0\) for \(j<i\).
	For a component represented by a series based at \(\bv\in P_i\), the \(i\)-th coordinate under \(\Lambda_C^{-1}\) is its coefficient at \(\bv\), and by the definition of the intrinsic coefficient spaces these coefficients span \(\cI_{P_i}^{L}\) as \(\bv\) ranges over \(P_i\) and \(\cN\) over the ordered families, so \(\im\mathsf T_{C,ii} = \cI_{P_i}^{L}\), which proves \eqref{eq:formal-transition-diagonal}.
	The subspace \(W_k\) of the lemma is \(\cW_{P_k}\) and its maps \(\mathsf T_{<k}\) and \(\mathsf T_{k,<k}\) are \(\mathsf T_{C,<k}\) and \(\mathsf T_{C,k,<k}\), so the lemma and \eqref{eq:formal-transition-diagonal} give \eqref{eq:formal-transition-associated-graded-image} for \(k\geq2\) and \(\cW_{P_1}=\im\mathsf T_{C,11}=\cI_{P_1}^{L}\) for \(k=1\).
\end{proof}

Put \(\delta_{P_k}^{L} = \sum_{\bv\in P_k}\delta_{\bv}^{L}\), and define \(\tau_{C,1}=0\) and
\begin{equation}
	\tau_{C,k}
	=
	\dim_{\C}
	\frac{
		\cI_{P_k}^{L}
		+
		\mathsf T_{C,k,<k}(\ker\mathsf T_{C,<k})}
	{\cI_{P_k}^{L}}
	\qquad(k\geq2).
	\label{eq:formal-transition-correction}
\end{equation}

The preceding theorem gives a codimension formula.

\begin{corollary}
	\label{cor:formal-solution-exact-codimension}
	We have \(0\leq\tau_{C,k}\leq\delta_{P_k}^{L}\), the number \(\tau_{C,k}\) is independent of the sections in \eqref{eq:formal-transition-lifting-map}, and
	\begin{equation}
		\dim_{\C}
		\left(
		\frac{\cV_{\bw}}
		{\cV_{\bw}^{L}}
		\right)
		=
		\sum_{\bv\in\sE_{\bbeta,\bw}}
		\delta_{\bv}^{L}
		-
		\sum_C\sum_{k=2}^{m(C)}\tau_{C,k}.
		\label{eq:formal-solution-exact-transition-codimension}
	\end{equation}
	Equivalently, the contribution of \(C\) to the codimension in \eqref{eq:formal-solution-exact-transition-codimension} is
	\begin{equation}
		\sum_{k=1}^{m(C)}\dim_{\C}\cE_{P_k}
		-
		\rank\mathsf T_C.
		\label{eq:formal-solution-cokernel-rank}
	\end{equation}
	The series obtained by intrinsic perturbation span \(\cV_{\bw}\) if and only if, for every \(L\)-coset \(C\), we have \(\delta_{P_1}^{L}=0\) and \(\tau_{C,k}=\delta_{P_k}^{L}\) for every \(k\geq2\).
	Equality holds in the upper bound in \eqref{eq:formal-solution-codimension-bounds} if and only if \(\tau_{C,k}=0\) for every \(C\) and \(k\geq2\).
	Equality holds in the lower bound if and only if \(\tau_{C,k}=\delta_{P_k}^{L}\) for every \(C\) and \(k\geq2\).
\end{corollary}

\begin{proof}
	Theorem~\ref{thm:formal-solution-exact-cokernel} gives \(\tau_{C,k}=\dim_{\C}(\cW_{P_k}/\cI_{P_k}^{L})\), which proves the independence assertion and yields \(\dim_{\C}(\cE_{P_k}/\cW_{P_k})=\delta_{P_k}^{L}-\tau_{C,k}\).
	Summing over the classes and using \eqref{eq:formal-filtration-exact-codimension} proves \eqref{eq:formal-solution-exact-transition-codimension}, the exact sequence \eqref{eq:formal-transition-cokernel} gives \eqref{eq:formal-solution-cokernel-rank}, and the three equivalences follow term by term.
\end{proof}

\begin{remark}
	For \(k\geq2\), the number \(\tau_{C,k}\) is the dimension of the image of \(\mathsf T_{C,k,<k}(\ker\mathsf T_{C,<k})\) in \(\cE_{P_k}/\cI_{P_k}^{L}\).
	Equivalently,
	\[
		\tau_{C,k}
		=
		\dim_{\C}
		\frac{
			\mathsf T_{C,k,<k}(\ker\mathsf T_{C,<k})}
		{
			\mathsf T_{C,k,<k}(\ker\mathsf T_{C,<k})
			\cap
			\cI_{P_k}^{L}}.
	\]
	The number \(\tau_{C,k}\) is the dimension, modulo \(\cI_{P_k}^{L}\), of the space of coefficient tuples indexed by \(P_k\) that come from certain linear combinations of series obtained by intrinsic perturbation.
	These combinations use outputs with exponents in the classes \(P_i\), \(i<k\), as the bases of their constructions, and their coordinates indexed by \(P_1,\ldots,P_{k-1}\) vanish.
	Theorem~\ref{thm:formal-solution-exact-cokernel} identifies the displayed quotient with \(\cW_{P_k}/\cI_{P_k}^{L}\), so \(\tau_{C,k}\) does not depend on the sections \(\sigma_{C,i}\).
\end{remark}

\begin{remark}
	The formula in \eqref{eq:formal-solution-cokernel-rank} is exact, but it is not a closed combinatorial formula.
	Once bases of the finite-dimensional spaces \(\im\sF_{\bv,\cN}\) and \(\cE_{P_k}\), together with the sections \(\sigma_{C,k}\), have been constructed, the blocks of \(\mathsf T_C=\Lambda_C^{-1}\Sigma_C\) are determined by finitely many coefficient polynomials, and the ranks of the matrices \(\mathsf T_C\) give the codimension.
	Thus the formula provides a finite-dimensional presentation of the quotient.
	The construction of these bases from the perturbation inputs requires a separate justification, since \(\cP_{\bv,\cN}^{\perp}\) can be infinite-dimensional.
\end{remark}

Thus, under the standing homogeneity assumption, the finite-dimensional maps \(\mathsf T_C\) determine both the exact codimension and whether the series obtained by intrinsic perturbation span \(\cV_{\bw}\).

\begin{corollary}
	\label{cor:formal-solution-completeness}
	If \(\delta_{\bv}^{L}=0\) for every \(\bv\in\sE_{\bbeta,\bw}\), then \(\cV_{\bw}^{L}=\cV_{\bw}\).
	If \(\delta_{\bv}^{L}>0\) for some \(P\in\sP_{\bbeta,\bw}^{\min}\) and some \(\bv\in P\), then the inclusion \(\cV_{\bw}^{L}\subseteq\cV_{\bw}\) is strict.
	If \(\delta_{\bv}^{L}=0\) whenever \(\bv\in P\) and \(P\notin\sP_{\bbeta,\bw}^{\min}\), then
	\[
		\dim_{\C}
		\left(
		\cV_{\bw}/\cV_{\bw}^{L}
		\right)
		=
		\sum_{P\in\sP_{\bbeta,\bw}^{\min}}
		\sum_{\bv\in P}
		\dim_{\C}
		\left(
		\frac{J_{\bv}^{L}}
		{\cE_{\bv}^{\perp}}
		\right).
	\]
\end{corollary}

\begin{proof}
	All three assertions follow from Theorem~\ref{thm:formal-solution-codimension-bounds} and Proposition~\ref{prop:all-ordered-local-quotient}.
\end{proof}

We use the following notation in the lemma on finite linear combinations of basic Nilsson solutions.
Suppose that, on a fixed branch of the logarithms, a formal series has the form \(\phi(e^{\by}) = \sum_{\balpha\in S}e^{\balpha\cdot\by}p_{\balpha}(\by)\) and \(p_{\balpha}(\by) = \sum_{|\bk|\leq M}c_{\balpha,\bk}\by^{\bk}\), where the exponents in \(S\) are distinct and \(M\) is independent of \(\balpha\).
For a series in this form, write \(\supp(\phi)=\{\balpha\in S\mid p_{\balpha}\ne0\}\).

\begin{lemma}
	\label{lem:nilsson-asymptotic-injectivity}
	Let \(\phi\) be a finite linear combination of basic Nilsson solutions on a fixed branch of the logarithms, and let \(M\) be a nonnegative integer bounding the degrees of their logarithmic coefficients.
	Let \(\bw''\) be an integral vector in the interior of a strongly convex open cone whose dual contains the supports of these solutions.
	Let \(\cU\subseteq\C^n\) be an open domain in the variables \(\by=\log\bx\) on which the coordinatewise exponential map is injective.
	Suppose that \(\sum_{\balpha\in S} \sum_{|\bk|\leq M} \bigl|c_{\balpha,\bk}e^{\balpha\cdot\by}\bigr| < \infty\) for every \(\by\in\cU\).
	This condition is pointwise in \(\by\), and it does not require uniform convergence on compact subsets of \(\cU\).
	Let \(\cB\subseteq\C^n\) be a nonempty open set, and suppose that a positive number \(t_0\) satisfies \(\bzeta+\bw''\log t\in\cU\) for \(\bzeta\in\cB,\ 0<t<t_0\).
	For \(\bzeta\in\cB\), put \(\bz=e^{\bzeta}\) and define the curve \(\bx(t)\) by \(x_j(t)=z_jt^{w_j''}\) for \(0<t<t_0\).
	Let \(\mu\) be the least real \(\bw''\)-weight of an exponent of \(\phi\), define \(\psi\) to be the finite sum of the terms whose exponents have real \(\bw''\)-weight \(\mu\), and suppose that every other exponent has real \(\bw''\)-weight at least \(\mu+\varepsilon\) for a positive number \(\varepsilon\).
	Then, for every \(\bzeta\in\cB\) and every \(t_*\in(0,t_0)\), \(t^{-\mu}\bigl(\phi(\bx(t))-\psi(\bx(t))\bigr) = O\bigl(t^{\varepsilon}(1+|\log t|)^M\bigr)\) as \(t\to0^+\) with \(0<t\leq t_*\).
	If \(\psi\ne0\), there is a choice of \(\bzeta\in\cB\), and hence of \(\bz=e^{\bzeta}\), for which the sum of \(\phi\) is not identically zero on the curve \(\bx(t)\).
\end{lemma}

\begin{proof}
	Write \(\phi(\bx) = \sum_{\balpha\in S}\bx^{\balpha}p_{\balpha}(\log\bx)\) and \(p_{\balpha}(\by) = \sum_{|\bk|\leq M}c_{\balpha,\bk}\by^{\bk}\).
	Fix \(\bzeta\in\cB\) and \(t_*\in(0,t_0)\), and use \(\bz^{\balpha}=e^{\balpha\cdot\bzeta}\) on the fixed branch.
	Define \(S_*\) by \(S_*=\sum_{\balpha,\bk}|c_{\balpha,\bk}\bz^{\balpha}|t_*^{\operatorname{Re}(\bw''\cdot\balpha)-\mu}\).
	By the hypothesis at the single point \(\by=\bzeta+\bw''\log t_*\), the number \(S_*\) is finite.
	For an exponent \(\balpha\notin\supp(\psi)\), put \(\delta_{\balpha} = \operatorname{Re}(\bw''\cdot\balpha)-\mu\), so that \(\delta_{\balpha}\geq\varepsilon\) and \(t^{\delta_{\balpha}} = t^\varepsilon t^{\delta_{\balpha}-\varepsilon} \leq t^\varepsilon t_*^{\delta_{\balpha}-\varepsilon}\) for \(0<t\leq t_*\).
	There is a constant \(c_0\), depending only on \(\bzeta\), \(\bw''\), and \(M\), with \(|(\bzeta+\bw''\log t)^{\bk}| \leq c_0(1+|\log t|)^M\) for \(|\bk|\leq M\) and \(0<t\leq t_*\).
	Consequently,
	\[
		\begin{aligned}
			\left|
			t^{-\mu}\bigl(\phi(\bx(t))-\psi(\bx(t))\bigr)
			\right|
			&\leq
			c_0(1+|\log t|)^M
			\sum_{\balpha\notin\supp(\psi)}
			\sum_{|\bk|\leq M}
			\bigl|c_{\balpha,\bk}\bz^{\balpha}\bigr|
			t^{\delta_{\balpha}}\\
			&\leq
			c_0t_*^{-\varepsilon}S_*
			t^\varepsilon(1+|\log t|)^M,
		\end{aligned}
	\]
	which proves the asserted estimate.

	Suppose that \(\psi\ne0\), and fix \(t_*\in(0,t_0)\).
	The set \(\cX_* = \exp(\cB+\bw''\log t_*) = t_*^{\bw''}\exp(\cB)\) is a nonempty open set on which the fixed branch of the logarithms is defined, and the finite exponential polynomial \(F(\by) = \psi(e^{\by}) = \sum_{\balpha}e^{\balpha\cdot\by}p_{\balpha}(\by)\) does not vanish identically on any nonempty open subset of \(\C^n\).
	Suppose, on the contrary, that \(F\) vanishes identically on some nonempty open set.
	Choose \(\bxi\in\C^n\) so that the numbers \(\balpha\cdot\bxi\) are pairwise distinct, and restrict \(F\) to the lines \(\by_0+\upsilon\bxi\).
	The functions \(e^{(\balpha\cdot\bxi)\upsilon}\) are linearly independent over \(\C[\upsilon]\), as we see by applying the products of the operators \((d/d\upsilon-\balpha'\cdot\bxi)^{M+1}\) corresponding to the other exponents, so every polynomial \(p_{\balpha}(\by_0+\upsilon\bxi)\) vanishes.
	Varying \(\by_0\) in an open set would give \(p_{\balpha}=0\) for every \(\balpha\), contrary to \(\psi\ne0\).
	Thus there is a point \(\bx_*\in\cX_*\) with \(\psi(\bx_*)\ne0\).
	Every point of \(\cX_*\) has the form \(\bz t_*^{\bw''}\) with \(\bz\in\exp(\cB)\), so put \(\bz=\bx_*t_*^{-\bw''}\).
	The restriction of \(\psi\) to this curve is nonzero at \(t=t_*\) and is therefore not identically zero.

	Group the terms of \(t^{-\mu}\psi(\bx(t))\) with equal imaginary \(\bw''\)-weights and put \(\varrho=\log t\), so that \(t^{-\mu}\psi(\bx(t)) = \sum_{j=1}^{N}e^{i\tau_j\varrho}f_j(\varrho)\), where the real numbers \(\tau_j\) are distinct and the polynomials \(f_j\) are not all zero.
	Let \(d_{\psi}\) be the largest degree of the polynomials \(f_j\) and let \(c_j\) be the coefficient of \(\varrho^{d_{\psi}}\) in \(f_j\), so that \(\varrho^{-d_{\psi}}\sum_je^{i\tau_j\varrho}f_j(\varrho) = \sum_jc_je^{i\tau_j\varrho}+o(1)\) as \(\varrho\to-\infty\).
	On the other hand,
	\[
		\lim_{R\to\infty}
		\frac{1}{R}
		\int_{-2R}^{-R}
		\left|
		\sum_{j=1}^{N}c_je^{i\tau_j\varrho}
		\right|^2d\varrho
		=
		\sum_{j=1}^{N}|c_j|^2
		>
		0.
	\]
	It follows that \(t^{-\mu}\psi(\bx(t))\) does not tend to zero as \(t\to0^+\).
	If the sum of \(\phi\) vanished identically on the curve, the first part of the proof would give \(t^{-\mu}\psi(\bx(t))=-t^{-\mu}(\phi(\bx(t))-\psi(\bx(t)))\to0\), which is a contradiction.
\end{proof}

Write \(\Sing(H_A(\bbeta))\) for the singular locus of \(H_A(\bbeta)\) in \(\C^n\), that is, for the projection to \(\C^n\) of the characteristic variety of \(M_A(\bbeta)\) outside the zero section.
For a connected open set \(\Omega\) in the nonsingular locus of \(H_A(\bbeta)\), let \(\Sol_{\Omega}(H_A(\bbeta))\) denote the space of holomorphic solutions on \(\Omega\).
For a connected open set \(\Omega\) in the nonsingular locus on which every element of \(\cV_{\bw}\) converges after branches of \(\log x_1,\ldots,\log x_n\) have been fixed, let \(\Sigma_{\Omega}:\cV_{\bw}\to\Sol_{\Omega}(H_A(\bbeta))\) denote the summation map, and put \(\cS_{\bw}^{L}(\Omega)=\Sigma_{\Omega}(\cV_{\bw}^{L})\).

\begin{theorem}
	\label{thm:analytic-realization}
	There is a nonempty simply connected open set \(\Omega \subseteq (\C^*)^n\setminus\Sing(H_A(\bbeta))\) and a choice of the branches of \(\log x_1,\ldots,\log x_n\) on \(\Omega\) such that every element of \(\cV_{\bw}\) converges on \(\Omega\) and summation gives an isomorphism
	\begin{equation}
		\Sigma_{\Omega}:
		\cV_{\bw}
		\xrightarrow{\ \sim\ }
		\Sol_{\Omega}(H_A(\bbeta)).
		\label{eq:analytic-summation-isomorphism}
	\end{equation}
	Then \(\Sigma_{\Omega}\) induces an isomorphism
	\begin{equation}
		\frac{\cV_{\bw}}
		{\cV_{\bw}^{L}}
		\xrightarrow{\ \sim\ }
		\frac{\Sol_{\Omega}(H_A(\bbeta))}
		{\cS_{\bw}^{L}(\Omega)}.
		\label{eq:analytic-intrinsic-quotient}
	\end{equation}
	In particular, the formula in Corollary~\ref{cor:formal-solution-exact-codimension} determines the codimension of \(\cS_{\bw}^{L}(\Omega)\) in the holomorphic solution space, and Theorem~\ref{thm:formal-solution-codimension-bounds} gives upper and lower bounds for this codimension.
\end{theorem}

\begin{proof}
	As in the proof of Proposition~\ref{prop:canonical-nilsson-identification}, we may replace \((A,\bbeta)\) by \((T^{-1}A,T^{-1}\bbeta)\), where the columns of \(T\in\Z^{d\times d}\) form a \(\Z\)-basis of \(\Z A\), retain the notation \((A,\bbeta)\), and assume \(\Z A=\Z^d\), as required in \cite{DMM12}.
	This replacement changes neither \(L\), nor the canonical series, nor the solution spaces.

	Choose a linear functional \(\bh\) such that \(\bh\ba_j=1\) for \(j=1,\ldots,n\), as permitted by the standing homogeneity assumption.
	Then \(\bone=\bh A\) belongs to \(\operatorname{rowspan}_{\R}(A)\), and the cone spanned by the columns of \(A\) is strongly convex.
	For \(\bu\in L\), we have \(|\bu| = \bh(A\bu) =0\).

	Choose \(\bw'\) as constructed in the proof of Proposition~\ref{prop:canonical-nilsson-identification}.
	The construction shows that \(\bw'\) is a perturbation of \(\bone\) in the sense of \cite[Definition~3.3]{DMM12}.
	Proposition~\ref{prop:canonical-nilsson-identification} identifies \(\cV_{\bw}\) with the span of the basic Nilsson solutions in the direction \(\bw'\), and \cite[Lemma~2.8]{DMM12} supplies bases with linearly independent initial series.
	For every basic Nilsson solution in \(\cV_{\bw}\), each occurring shift belongs to \(L\) and has coordinate sum zero, so \cite[Theorem~6.2]{DMM12} applies.
	By \cite[Theorem~6.4]{DMM12}, there is a common nonempty open domain of convergence on which every element of \(\cV_{\bw}\) converges.
	Choose a \(\Z\)-basis \(\{\bgamma_1,\ldots,\bgamma_{n-d}\}\) of \(L\) such that \(\bgamma_i\cdot\bw'>0\) for \(i=1,\ldots,n-d\), as in \cite[Notation~6.1]{DMM12}.
	The intersection of the open cone that defines the direction \(\bw'\) with the open cone of vectors having a positive dot product with each \(\bgamma_i\) is a strongly convex open rational polyhedral cone contained in \(\R_{>0}^n\); it contains \(\bw'\), and hence we can choose a positive integral vector \(\bw''\) in this intersection.
	For this basis, the common domain of convergence contains a set defined as in \cite[Notation~6.1]{DMM12} by finitely many inequalities of the form \(|\bx^{\bgamma_i}|<\varepsilon_i\), and the choice of \(\bw''\) gives \(\bgamma_i\cdot\bw''>0\) for \(i=1,\ldots,n-d\).
	In the coordinates \(\log\bx\), these inequalities cut out open half-spaces.
	Choose a nonempty bounded open box \(\cB\subseteq\C^n\) whose projection onto the imaginary part of each coordinate has length less than \(2\pi\).
	Choose a sufficiently negative real number \(\upsilon_0\) so that the tube \(\cB+\{\upsilon\bw''\mid \upsilon<\upsilon_0\}\) lies inside the set defined by these inequalities and hence inside the common domain of convergence.
	Define \(\Omega'\) to be the image of this tube under the coordinatewise exponential map.
	The exponential map is injective on the tube, so \(\Omega'\) is a nonempty simply connected open subset of the common domain of convergence, the branches of the logarithms are fixed on \(\Omega'\), and \(\Omega'\) contains the tail of the curve defined by \(x_j(t)=z_jt^{w_j''}\) whenever \(\log\bz\in\cB\).
	By \cite[Corollary~2.4.16]{SST00}, equivalently by the dimension assertion in \cite[Theorem~6.4]{DMM12}, we have \(\dim_{\C}\cV_{\bw}=\rank(H_A(\bbeta))\).

	We next prove that summation on \(\Omega'\) is injective.
	Take \(0\ne\phi\in\cV_{\bw}\).
	By \cite[Proposition~2.5.2]{SST00}, the series \(\phi\) has a nonzero finite initial series with respect to \(\bw''\); define \(\psi\) to be this initial series, and let \(\mu\) be the common real part of the \(\bw''\)-weights of the exponents of \(\psi\).
	Since \(\phi\) is a finite linear combination of basic Nilsson solutions whose supports lie in the dual of a strongly convex open cone having \(\bw''\) in its interior, there is a positive number \(\varepsilon\) such that every term whose exponent lies outside \(\supp(\psi)\) has real \(\bw''\)-weight at least \(\mu+\varepsilon\).
	The domain in \cite[Notation~6.1]{DMM12} is defined by inequalities in the moduli of the coordinates, so absolute convergence at a point of \(\Omega'\) implies absolute convergence at every point obtained by shifting the logarithms by an element of \((2\pi i\Z)^n\).
	Let \(M\) be a common bound for the degrees of the logarithmic coefficients in \(\phi\), and put \(\cT = \cB+\{\upsilon\bw''\mid \upsilon<\upsilon_0\}\).
	Fix \(\by\in\cT\), and write \(\phi(e^{\by}) = \sum_{\balpha\in S} e^{\balpha\cdot\by}p_{\balpha}(\by)\) and \(p_{\balpha}(\by) = \sum_{|\bk|\leq M}c_{\balpha,\bk}\by^{\bk}\).
	Put \(K_M = \{\bk\in\N^n\mid |\bk|\leq M\}\) and \(N = |K_M| = \binom{n+M}{n}\), and enumerate \(K_M\) as \(\bk^{(1)},\ldots,\bk^{(N)}\).
	Let \(\by'=(y'_1,\ldots,y'_n)\) be an auxiliary vector of variables.
	Define \(\operatorname{ev}_{\by,M}:\C[\by']_{\leq M}\to\C^N\) by \(\operatorname{ev}_{\by,M}(p)=\bigl(p(\by+2\pi i\bk^{(1)}),\ldots,p(\by+2\pi i\bk^{(N)})\bigr)\).
	The map \(\operatorname{ev}_{\by,M}\) is injective: after the invertible change of variables \(\widetilde p(\by')=p(\by+2\pi i\by')\), the polynomials \(\prod_{\nu=1}^{n}\binom{y'_\nu}{k_\nu}\) with \(\bk\in K_M\) form a basis whose evaluation matrix on \(K_M\), ordered by total degree, is triangular with diagonal entries equal to one.
	Since the domain and codomain of \(\operatorname{ev}_{\by,M}\) both have dimension \(N\), this map is an isomorphism.
	Equip \(\C[\by']_{\leq M}\) with the coefficient \(\ell^1\)-norm.
	Every polynomial \(p(\by')=\sum_{|\bk|\leq M}c_{\bk}(\by')^{\bk}\) then satisfies \(\sum_{|\bk|\leq M}|c_{\bk}|\leq\|\operatorname{ev}_{\by,M}^{-1}\|_{\ell^\infty\to\ell^1}\max_{1\leq j\leq N}|p(\by+2\pi i\bk^{(j)})|\).

	Replacing \(\log\bx\) by \(\log\bx+2\pi i\bk^{(j)}\) preserves the supports and the logarithmic degree bounds and sends every basic Nilsson solution to another basic Nilsson solution in the same direction.
	On the common domain of convergence in \cite[Theorem~6.4]{DMM12}, the proof of \cite[Theorem~6.2]{DMM12} gives \(\sum_{\balpha\in S} \left| e^{\balpha\cdot(\by+2\pi i\bk^{(j)})} p_{\balpha}(\by+2\pi i\bk^{(j)}) \right| < \infty\) for \(j=1,\ldots,N\).
	Since \(\phi\) is a finite linear combination of basic Nilsson solutions and every shift in each of those solutions belongs to \(L\subseteq\Z^n\), only finitely many imaginary parts occur among the exponents in \(S\).
	Hence the number \(D_{\phi,M}=\max_{\balpha\in S,\ \bk\in K_M}\exp\bigl(2\pi\operatorname{Im}(\balpha)\cdot\bk\bigr)\) is finite.
	For \(\balpha\in S\) and \(\bk\in K_M\), we have \(|e^{\balpha\cdot\by}|=\exp(2\pi\operatorname{Im}(\balpha)\cdot\bk)|e^{\balpha\cdot(\by+2\pi i\bk)}|\leq D_{\phi,M}|e^{\balpha\cdot(\by+2\pi i\bk)}|\).
	Applying the inequality for \(\operatorname{ev}_{\by,M}^{-1}\) to each \(p_{\balpha}\) and then summing over \(\balpha\) gives
	\begin{multline*}
		\sum_{\balpha\in S}
		\sum_{|\bk|\leq M}
		\left|
		c_{\balpha,\bk}e^{\balpha\cdot\by}
		\right|
		\\
		\leq
		\|\operatorname{ev}_{\by,M}^{-1}\|_{\ell^\infty\to\ell^1}D_{\phi,M}
		\sum_{j=1}^{N}
		\sum_{\balpha\in S}
		\left|
		e^{\balpha\cdot(\by+2\pi i\bk^{(j)})}
		p_{\balpha}(\by+2\pi i\bk^{(j)})
		\right|
		<
		\infty.
	\end{multline*}
	Since \(\by\in\cT\) was arbitrary, we have \(\sum_{\balpha\in S} \sum_{|\bk|\leq M} \bigl|c_{\balpha,\bk}e^{\balpha\cdot\by}\bigr| < \infty\) for every \(\by\in\cT\).
	Lemma~\ref{lem:nilsson-asymptotic-injectivity} shows that the sum of \(\phi\) is not identically zero on \(\Omega'\).
	Thus summation on \(\Omega'\), and hence on every nonempty open subset of \(\Omega'\), is injective.

	The domain \(\Omega'\) is nonempty and open in the usual topology, and \(\Sing(H_A(\bbeta))\) is a proper algebraic subset, so \(\Omega'\cap\bigl((\C^*)^n\setminus\Sing(H_A(\bbeta))\bigr)\) is nonempty.
	Choose a simply connected open set \(\Omega\) in this intersection.
	The holomorphic solution space on \(\Omega\) has dimension \(\rank(H_A(\bbeta))\), so the summation map is an injective map between two spaces of the same finite dimension and is therefore surjective, which gives \eqref{eq:analytic-summation-isomorphism}.
	Restricting the isomorphism to \(\cV_{\bw}^{L}\) and passing to quotients proves \eqref{eq:analytic-intrinsic-quotient} and the last assertion.
\end{proof}

\section{Two families of examples}
\label{sec:applications}

In this section, we apply the preceding results to the five-column configuration of Theorem~5.1 of \cite{Nak26} and to the family of lattice rank 2 of Theorem~6.2 of \cite{Nak26}.
We recall from \cite{Nak26} the data that the computations below use.

For the five-column configuration, we use
\[
	A=
	\begin{pmatrix}
		1&1&1&1&1\\
		2&0&0&3&3\\
		3&3&2&0&3
	\end{pmatrix},
	\qquad
	B=
	\begin{pmatrix}
		3&0\\
		2&3\\
		-3&-3\\
		1&1\\
		-3&-1
	\end{pmatrix}.
\]
We put \(\bv=\tp{(-1,0,-1,0,0)}\), \(\bbeta=A\bv=\tp{(-2,-2,-5)}\), \(\vartheta=40+\sqrt2\), and \(\bw=(-(2\vartheta+3)/9,\vartheta/3,0,0,0)\).
Thus \(\bw B=(-1,\vartheta)\).
We put \(\cN^{\circ}=\{\{1,3\},\{1,4\}\}\) and use \(\cN=\cN^{\circ}\) in \eqref{eq:five-column-ideals}.
For this configuration, the linear forms are \((\ell_1,\ldots,\ell_5)=(3s_1,2s_1+3s_2,-3s_1-3s_2,s_1+s_2,-3s_1-s_2)\), and the two ideals are
\begin{equation}
	J_{\cN}^{\amb}=\langle\ell_5,\ell_2\ell_4\rangle=\langle3s_1+s_2,(s_1+s_2)^2\rangle, \qquad J_{\cN}^{\intr}=\langle\ell_2,\ell_5\rangle=\langle s_1,s_2\rangle.
	\label{eq:five-column-ideals}
\end{equation}
We list the members of \(\sS(\bv)\) for the five-column configuration and indicate which support fibers \(\cF_I\) are contained in \(\cC(\bw)\).
\begin{equation}
	\begin{array}{c|c}
		I&\text{semigroup containment}\\ \hline
		\{1,3\}&\cF_I\subseteq\cC(\bw)\\
		\{1,4\}&\cF_I\subseteq\cC(\bw)\\
		\{1,3,5\}&\cF_I\subseteq\cC(\bw)\\ \hline
		\{2,4\},\ \{1,2,4\},\ \{3,5\},\ \{2,3,5\},\ \{2,4,5\}&\text{negative weight}.
	\end{array}
	\label{eq:five-column-supports}
\end{equation}
Hence \(\cN_{\bv}=\{\{1,3\},\{1,4\},\{1,3,5\}\}\) and \(K_{\cN_{\bv}}=\{1\}\).
Write \(J_0=\{1,3\}\), \(J_1=\{1,4\}\), and \(J_2=\{1,3,5\}\) for the three members of this collection.

Fix an integer \(q\geq1\) for the family of lattice rank 2.
The matrix \(B_q^\triangle\) is obtained by stacking \(q\) copies of \(B\), and \(\bw_q^\triangle\) is obtained by concatenating \(q\) copies of \(\bw/q\).
Thus \(\bw_q^\triangle B_q^\triangle=(-1,\vartheta)\).
We use the homogeneous matrix \(A_q^\triangle\) of \cite[Theorem~6.2]{Nak26}, which has rank \(5q-2\) and satisfies \(\ker_{\Z}(A_q^\triangle)=\im_{\Z}(B_q^\triangle)\).
Label the polynomial variables by \(t_{ij}\), where \(1\leq i\leq q\) and \(1\leq j\leq5\), and put \(T_j=\prod_{i=1}^q t_{ij}\).
Then \(\Phi_0(T_j)=\ell_j^q\).
For the family of lattice rank 2, put \(\lambda_1=2s_1+3s_2\), \(\lambda_2=3s_1+s_2\), and \(\lambda_3=s_1+s_2\) in \(\C[[s_1,s_2]]\), and write \(J_k^\triangle\) for the union of the \(q\) indexed copies of \(J_k\).
The fake exponent \(\bv_q^\triangle\) repeats \(\bv_0=\tp{(-1,0,-1,0,0)}\) in the \(q\) coordinate blocks, the ordered negative support family is \(\cN_q^{\triangle,\circ}=\{J_0^\triangle,J_1^\triangle\}\), and the two ideals are
\begin{equation}
	\Phi\bigl(Q_{\cN_q^{\triangle,\circ}}(\bt):e_q^\triangle\bigr)=\langle\lambda_2^q,\lambda_1^q\lambda_3^q\rangle, \qquad \Pi_{\cN_q^{\triangle,\circ}}:m_{\bv_q^\triangle,\cN_q^{\triangle,\circ}}=\langle\lambda_1^q,\lambda_2^q\rangle.
	\label{eq:family-ideals}
\end{equation}
\begin{corollary}
	\label{cor:lattice-counterexample-formal-codimension}
	For the system in Theorem~5.1 of \cite{Nak26}, \(\dim_{\C}\cV_{\bw}=12\), \(\dim_{\C}\cV_{\bw}^{L}=11\), and \(\dim_{\C} \left( \cV_{\bw}/\cV_{\bw}^{L} \right) =1\).
	There is a nonempty simply connected open set \(\Omega\), contained in the nonsingular locus, on which the canonical series converge, and on this set we have
	\[
		\begin{aligned}
			\dim_{\C}\Sol_{\Omega}(H_A(\bbeta))
			&=12,
			&
			\dim_{\C}\cS_{\bw}^{L}(\Omega)
			&=11,\\
			\dim_{\C}
			\frac{\Sol_{\Omega}(H_A(\bbeta))}
			{\cS_{\bw}^{L}(\Omega)}
			&=1.
		\end{aligned}
	\]
	A nonzero class in \(\cV_{\bw}/\cV_{\bw}^{L}\) has a representative whose \(\bw\)-initial term is
	\[
		\begin{aligned}
			&\bx^{\bv}
			\left(
			(\bb^{(1)}-3\bb^{(2)})\cdot\log\bx
			\right)\\
			&\qquad
			=x_1^{-1}x_3^{-1}
			\left(
			3\log x_1-7\log x_2+6\log x_3-2\log x_4
			\right).
		\end{aligned}
	\]
\end{corollary}

\begin{proof}
	The twelve standard pairs of \(\inw(I_A)\) give eleven distinct fake exponents: \(\bv\) occurs twice, and each of the other ten occurs once.
	In the coordinates \(\bv+B\bs\), the local component of the fake indicial ideal at \(\bs=\bzero\) is \(\langle3s_1+s_2,(s_1+s_2)^2\rangle\), which has length two, and the other ten local components are reduced.
	The inverse system of this component is \(\Span\{1,\partial_{s_1}-3\partial_{s_2}\}\), and each of the other ten inverse systems is spanned by \(1\).
	All eleven fake exponents have minimal negative support.
	At \(\bv\), the only ordered negative support families are \(\{\{1,3\}\}\) and \(\cN^{\circ}\), and both intrinsic coefficient spaces are \(\C\).
	Hence \(\cN_{\bv}^{\ord}=\cN^{\circ}\), and \eqref{eq:five-column-ideals} and Lemma~\ref{lem:largest-ordered-family} give \(\cE_{\bv}=\Span\{1,\partial_{s_1}-3\partial_{s_2}\}\).
	It follows that \(\cI_{\bv}^{L}=\C\) and \(\delta_{\bv}^{L}=1\).
	For each of the other ten fake exponents \(\widetilde{\bv}\), the singleton family \(\{I_{\bzero}\}\) produces the constant coefficient, and the inclusions in \eqref{eq:actual-fake-local-inclusions} show that \(\cE_{\widetilde{\bv}}\) is contained in the one-dimensional inverse system \(P_B(\widetilde{\bv})^\perp\).
	Hence \(\cE_{\widetilde{\bv}}=\cI_{\widetilde{\bv}}^{L}=\C\) and \(\delta_{\widetilde{\bv}}^{L}=0\) for each of these ten fake exponents.
	Thus all eleven fake exponents are exponents, or equivalently, they belong to \(\sE_{\bbeta,\bw}\).
	The only two exponents lying in the same \(L\)-coset are \(\bv\) and \(\tp{(-10,0,2,-1,7)}=\bv+B\tp{(-3,2)}\).
	Since \(\bw B=(-1,\vartheta)\) and \(\vartheta\) is irrational, \(\bw\) is nonzero on every nonzero element of \(L\), so the classes in \(\sP_{\bbeta,\bw}\) are singletons.
	The inequality \(\bw\cdot\bigl(B\tp{(-3,2)}\bigr)=3+2\vartheta>0\) shows that \(\{\bv\}\in\sP_{\bbeta,\bw}^{\min}\).
	The lower and upper bounds in Theorem~\ref{thm:formal-solution-codimension-bounds} are therefore both equal to one.
	The first identity in \eqref{eq:formal-filtration-dimensions} gives \(\dim_{\C}\cV_{\bw}=2+10=12\), and hence \(\dim_{\C}\cV_{\bw}^{L}=11\).
	The displayed initial term represents the nonzero local quotient at \(\bv\), so the corresponding canonical series gives a nonzero class in \(\cV_{\bw}/\cV_{\bw}^{L}\).
	The analytic assertions follow from Theorem~\ref{thm:analytic-realization}.
	The ancillary script \texttt{verify\_rank\_two\_solution\_codimension.sage} checks, in exact arithmetic, the fake exponents obtained from the listed standard pairs, the local lengths, the minimal negative supports, the ordered negative support families within the supplied distinguished collection at \(\bv\), and the lower and upper bounds giving codimension one; the ancillary script \texttt{verify\_rank\_two\_solution\_codimension.m2} computes the standard pairs and independently checks the primary decomposition of the fake indicial ideal and the local lengths.
\end{proof}
\begin{lemma}
	\label{lem:three-linear-form-powers-length}
	Let \(\lambda_1,\lambda_2,\lambda_3\) be pairwise nonproportional linear forms in \(S_0=\C[s_1,s_2]\).
	For every integer \(q\geq1\),
	\begin{equation}
		\dim_{\C}
		\frac{S_0}{\langle\lambda_1^q,\lambda_2^q,\lambda_3^q\rangle}
		=
		\left\lceil\frac{3q^2}{4}\right\rceil.
		\label{eq:three-linear-form-powers-length}
	\end{equation}
\end{lemma}

\begin{proof}
	After a linear change of variables and two nonzero rescalings, it is enough to prove the formula for \(\lambda_1=s_1\), \(\lambda_2=s_2\), and \(\lambda_3=s_1+s_2\).
	The formula is immediate when \(q=1\), so assume that \(q\geq2\).
	Put \(\overline S_q=S_0/\langle s_1^q,s_2^q\rangle\).
	The Hilbert function of \(\overline S_q\) is
	\[
		h_{\overline S_q}(d)
		=
		\begin{cases}
			d+1,    & 0\leq d\leq q-1, \\
			2q-1-d, & q\leq d\leq2q-2.
		\end{cases}
	\]
	The strong Lefschetz theorem for monomial complete intersections in characteristic zero \cite[Theorem~5 and Proposition~9]{RRR91} shows that multiplication by \((s_1+s_2)^q\) has maximal rank in every degree.
	For \(0\leq d\leq q-2\), the rank of multiplication by \((s_1+s_2)^q\) from \((\overline S_q)_d\) to \((\overline S_q)_{d+q}\) is therefore \(\min\{d+1,q-1-d\}\).
	Consequently,
	\begin{align*}
		\dim_{\C}\overline S_q/\langle(s_1+s_2)^q\rangle
		 & =q^2-
		\sum_{d=0}^{q-2}\min\{d+1,q-1-d\}          \\
		 & =
		\begin{cases}
			3q^2/4,     & q\equiv0\pmod2, \\
			(3q^2+1)/4, & q\equiv1\pmod2,
		\end{cases}                       \\
		 & =\left\lceil\frac{3q^2}{4}\right\rceil.
	\end{align*}
\end{proof}
We now compute the dimensions for the family of lattice rank 2.
\begin{theorem}
	\label{thm:fixed-rank-unbounded-solution-codimension}
	For every integer \(q\geq1\), consider the system in Theorem~6.2 of \cite{Nak26} and put \(\bbeta_q^\triangle=A_q^\triangle\bv_q^\triangle\).
	The vector \(\bv_q^\triangle\) is an exponent of \(H_{A_q^\triangle}(\bbeta_q^\triangle)\) with respect to \(\bw_q^\triangle\), its class is least in its \(L\)-coset, and
	\begin{equation}
		\delta_{\bv_q^\triangle}^{L}
		=
		\left\lceil\frac{3q^2}{4}\right\rceil.
		\label{eq:fixed-rank-local-all-family-codimension}
	\end{equation}
	The canonical formal solution space satisfies
	\begin{equation}
		\dim_{\C}\cV_{\bw_q^\triangle}=12q^2,
		\qquad
		\dim_{\C}
		\frac{\cV_{\bw_q^\triangle}}
		{\cV_{\bw_q^\triangle}^{L}}
		=
		\left\lceil\frac{3q^2}{4}\right\rceil.
		\label{eq:fixed-rank-formal-solution-codimension}
	\end{equation}
	For this system, the correction terms in Corollary~\ref{cor:formal-solution-exact-codimension} satisfy \(\tau_{C,k}=0\) for every \(L\)-coset \(C\) meeting \(\sE_{\bbeta_q^\triangle,\bw_q^\triangle}\) and every \(1\leq k\leq m(C)\).
	There is a nonempty simply connected open set \(\Omega_q\subseteq(\C^*)^{5q}\setminus\Sing(H_{A_q^\triangle}(\bbeta_q^\triangle))\) and a choice of the branches of the logarithms on \(\Omega_q\) such that summation identifies \(\cV_{\bw_q^\triangle}\) with \(\Sol_{\Omega_q}(H_{A_q^\triangle}(\bbeta_q^\triangle))\) and
	\begin{equation}
		\begin{aligned}
			\dim_{\C}\Sol_{\Omega_q}(H_{A_q^\triangle}(\bbeta_q^\triangle))
			&=12q^2,\\
			\dim_{\C}
			\frac{\Sol_{\Omega_q}(H_{A_q^\triangle}(\bbeta_q^\triangle))}
			{\cS_{\bw_q^\triangle}^{L}(\Omega_q)}
			&=\left\lceil\frac{3q^2}{4}\right\rceil.
		\end{aligned}
		\label{eq:fixed-rank-holomorphic-solution-codimension}
	\end{equation}
	Thus, even after all exponents occurring in the canonical formal solution space and all ordered negative support families are included, the codimension of \(\cS_{\bw_q^\triangle}^{L}(\Omega_q)\) is unbounded among homogeneous systems of lattice rank 2 with the other properties in Theorem~6.2 of \cite{Nak26}.
\end{theorem}

\begin{proof}
	We use \(J_0^\triangle,J_1^\triangle,J_2^\triangle\) and \(\lambda_1,\lambda_2,\lambda_3\) defined above.
	The two ideals in \eqref{eq:family-ideals} are the completions of the homogeneous polynomial ideals with the same generators, and below we use the contractions of these completions to \(S_0\).
	By the support table in \eqref{eq:five-column-supports}, repeated in the \(q\) blocks, the ordered negative support families for \(\bv_q^\triangle\) are exactly \(\{J_0^\triangle\}\) and \(\cN_q^{\triangle,\circ}=\{J_0^\triangle,J_1^\triangle\}\).
	A family containing \(J_2^\triangle\) is not ordered, because the union of the \(q\) indexed copies of \(\{3,5\}\) belongs to \(\sS(\bv_q^\triangle)\) and is a proper subset of \(J_2^\triangle\).
	Since \(J_{\cN_q^{\triangle,\circ}}^{\amb}=\langle\lambda_2^q,\lambda_1^q\lambda_3^q\rangle\) is proper and homogeneous, its inverse system contains the constant polynomial \(1\), which is a nonzero element of the ambient coefficient space at \(\bx^{\bv_q^\triangle}\).
	Theorem~5.5 of \cite{OS25}, applied to \(\cN_q^{\triangle,\circ}\), gives a series obtained by ambient perturbation in the direction \(\bw_q^\triangle\) whose coefficient at \(\bx^{\bv_q^\triangle}\) is nonzero.
	By Proposition~\ref{prop:canonical-nilsson-identification}, this series belongs to \(\cV_{\bw_q^\triangle}\).
	This series is supported in \(\bv_q^\triangle+L\), and every occurring shift \(\bu\) satisfies \(I_{\bu}\in\cN_q^{\triangle,\circ}\subseteq\cN_{\bv_q^\triangle}\), so \(\bw_q^\triangle\cdot\bu\geq0\).
	Thus the coefficient at \(\bx^{\bv_q^\triangle}\), written in lattice coordinates, is a nonzero element of \(\cE_{\bv_q^\triangle}\), and hence \(\bv_q^\triangle\in\sE_{\bbeta_q^\triangle,\bw_q^\triangle}\).
	Thus \(\cN_q^{\triangle,\circ}\) is the largest ordered family, and Lemma~\ref{lem:largest-ordered-family} and the first identity in \eqref{eq:family-ideals} give
	\begin{equation}
		\cE_{\bv_q^\triangle}^{\perp}
		=
		\langle\lambda_2^q,\lambda_1^q\lambda_3^q\rangle.
		\label{eq:fixed-rank-full-local-ideal}
	\end{equation}
	For the singleton family \(\{J_0^\triangle\}\), we have \(K_{\{J_0^\triangle\}}=J_0^\triangle\), \(e=1\), and \(M_{\{J_0^\triangle\}}=\langle1\rangle\), and the monomials \(\bt^{H}\) with \(H\in\cH_{\{J_0^\triangle\}}\) are \(T_4\) and \(T_5\), so the ideals \(J_{\{J_0^\triangle\}}^{\amb}\) and \(J_{\{J_0^\triangle\}}^{\intr}\) both equal
	\begin{equation}
		\langle\lambda_3^q,\lambda_2^q\rangle.
		\label{eq:fixed-rank-singleton-ideal}
	\end{equation}
	Proposition~\ref{prop:all-ordered-local-quotient}, the second identity in \eqref{eq:family-ideals}, and \eqref{eq:fixed-rank-singleton-ideal} now give
	\begin{equation}
		J_{\bv_q^\triangle}^{L}
		=
		\langle\lambda_1^q,\lambda_2^q\rangle
		\cap
		\langle\lambda_3^q,\lambda_2^q\rangle.
		\label{eq:fixed-rank-all-family-intrinsic-ideal}
	\end{equation}

	In the polynomial ring \(S_0\), define \(\cJ_1\), \(\cJ_3\), and \(\cJ_{13}\) by \(\cJ_1=\langle\lambda_1^q,\lambda_2^q\rangle\), \(\cJ_3=\langle\lambda_3^q,\lambda_2^q\rangle\), and \(\cJ_{13}=\langle\lambda_2^q,\lambda_1^q\lambda_3^q\rangle\).
	The three linear forms \(\lambda_1,\lambda_2,\lambda_3\) are pairwise nonproportional, so \(S_0/\cJ_1\) and \(S_0/\cJ_3\) both have length \(q^2\) and \(S_0/\cJ_{13}\) has length \(2q^2\).
	The standard exact sequence \(0\to S_0/(\cJ_1\cap\cJ_3)\to(S_0/\cJ_1)\oplus(S_0/\cJ_3)\to S_0/(\cJ_1+\cJ_3)\to0\) gives \(\dim_{\C}S_0/(\cJ_1\cap\cJ_3)=2q^2-\dim_{\C}S_0/(\cJ_1+\cJ_3)\).
	Using \(\cJ_{13}\subseteq\cJ_1\cap\cJ_3\), \(\dim_{\C}S_0/\cJ_{13}=2q^2\), and \(\cJ_1+\cJ_3=\langle\lambda_1^q,\lambda_2^q,\lambda_3^q\rangle\), we obtain \(\dim_{\C}\bigl((\cJ_1\cap\cJ_3)/\cJ_{13}\bigr)=\dim_{\C}S_0/\langle\lambda_1^q,\lambda_2^q,\lambda_3^q\rangle=\lceil3q^2/4\rceil\), where Lemma~\ref{lem:three-linear-form-powers-length} gives the last equality.
	Completion preserves these finite lengths, so \eqref{eq:fixed-rank-full-local-ideal}, \eqref{eq:fixed-rank-all-family-intrinsic-ideal}, and Proposition~\ref{prop:all-ordered-local-quotient} prove \eqref{eq:fixed-rank-local-all-family-codimension}.

	We next determine all fake exponents.
	For the five-column matrix \(A\) displayed above, put
	\[
		\begin{aligned}
			f_1&=\partial_3^3\partial_5^3-\partial_1^3\partial_2^2\partial_4,\\
			f_2&=\partial_2^3\partial_4-\partial_3^3\partial_5,\\
			f_3&=\partial_2\partial_5^2-\partial_1^3.
		\end{aligned}
	\]
	The first monomials have larger \(\bw\)-weight, the first two are relatively prime, and the other two \(S\)-polynomials satisfy \(S(f_1,f_3)=-\partial_1^3f_2\) and \(S(f_2,f_3)=-f_1\).
	Thus these binomials form a reduced Gr\"obner basis of the ideal \(J=\langle f_1,f_2,f_3\rangle\).
	Since no initial monomial is divisible by \(\partial_1\), we have \(J:\partial_1^\infty=J\).
	After we invert \(\partial_1\), the equation \(f_3=0\) makes \(\partial_2\) and \(\partial_5\) units, and \(f_2=0\) gives \(\partial_4=\partial_1^{-9}\partial_3^3\partial_5^7\).
	The localized quotient is therefore \(\C[\partial_1^{\pm1},\partial_5^{\pm1},\partial_3]\).
	The minor of \(A\) on columns one, three, and five is \(1\), so these three toric monomials are algebraically independent, and hence \(J=I_A\).
	Define \(\rho_q\) by \(\rho_q(\partial_j)=\prod_{i=1}^q\partial_{ij}\).
	After removal of a common monomial factor, every lattice binomial for \(A_q^\triangle\) is the image under \(\rho_q\) of a lattice binomial for \(A\).
	The map \(\rho_q\) preserves least common multiples, initial monomials, and the displayed reductions, so \(\rho_q(f_1),\rho_q(f_2),\rho_q(f_3)\) form the reduced Gr\"obner basis of \(I_{A_q^\triangle}\).
	We parametrize the solutions of \(A_q^\triangle\balpha=\bbeta_q^\triangle\) by \(\balpha_q(p,t)=\bv_q^\triangle+B_q^\triangle\tp{(p,t)}\), where \(p,t\in\C\).
	Each block of \(\balpha_q(p,t)\) has coordinates
	\[
		(\zeta_1,\zeta_2,\zeta_3,\zeta_4,\zeta_5)
		=(-1+3p,\,2p+3t,\,-1-3p-3t,\,p+t,\,-3p-t).
	\]
	The restrictions of the distractions of the three initial monomials are \(([\zeta_3]_3[\zeta_5]_3)^q\), \(([\zeta_2]_3\zeta_4)^q\), and \((\zeta_2[\zeta_5]_2)^q\), so the fake exponents are exactly the solutions of
	\begin{equation}
		[\zeta_3]_3[\zeta_5]_3=0,
		\qquad
		[\zeta_2]_3\zeta_4=0,
		\qquad
		\zeta_2[\zeta_5]_2=0.
		\label{eq:fixed-rank-same-coset-fake-exponents}
	\end{equation}
	If \(\zeta_2=0\), then \(t=-2p/3\), \(\zeta_3=-1-p\), and \(\zeta_5=-7p/3\), so the first equation gives \(p\in\{0,-1,-2,-3,-3/7,-6/7\}\).
	If \(\zeta_2\ne0\), then \(\zeta_5\in\{0,1\}\), the first equation holds automatically, and the second equation gives \(\zeta_2\in\{1,2\}\) or \(\zeta_4=0\).
	For \(\zeta_5=0\), these conditions give \((-1/7,3/7)\) and \((-2/7,6/7)\); the case \(\zeta_4=0\) gives \((0,0)\), which is excluded by \(\zeta_2\ne0\).
	For \(\zeta_5=1\), the conditions give \((-4/7,5/7)\), \((-5/7,8/7)\), and \((-1/2,1/2)\).
	Define the sets
	\[
		\begin{aligned}
			\cX_{13}&=\{(-a,2a/3)\mid a=1,2,3\},\\
			\cX_{12}&=\{(-3/7,2/7),(-6/7,4/7),(-1/7,3/7),\\
			&\qquad(-2/7,6/7),(-4/7,5/7),(-5/7,8/7)\},\\
			\cX_{23}&=\{(-1/2,1/2)\}.
		\end{aligned}
	\]
	Thus the parameter points of all fake exponents are exactly the eleven distinct points in \(\{(0,0)\}\cup\cX_{13}\cup\cX_{12}\cup\cX_{23}\).

	We now compute the singleton families at the ten points different from \((0,0)\).
	For \(J\subseteq\{1,\ldots,5\}\), write \(J^\triangle=\{(i,j)\mid1\leq i\leq q,\ j\in J\}\).
	At each of these ten points, put \(\balpha=\balpha_q(p,t)\) and \(I_{\bzero}=\nsupp(\balpha)\).
	Every lattice shift has the form \(B_q^\triangle\tp{(u,v)}\) with \(u,v\in\Z\), so every member of \(\sS(\balpha)\) is the union of \(q\) indexed copies of a subset of \(\{1,\ldots,5\}\).

	For \((p,t)=(-a,2a/3)\), with \(a=1,2,3\), the five coordinates are \((-1-3a,0,a-1,-a/3,7a/3)\).
	Put \(k=u+v\) for an integer shift.
	If the shifted second and third coordinates are nonnegative, then \(3k-u\geq0\) and \(a-1-3k\geq0\), so \(k\leq0\), \(u\leq3k\leq0\), and \(v\geq-2k\geq0\).
	The shifted first coordinate is then negative, the fourth coordinate \(-a/3+k\) is a negative integer only when \(a=3\), and the fifth coordinate \(7a/3-2u-k\) is positive.
	Hence the negative support is exactly \(I_{\bzero}\), which is \(\{1\}^\triangle\) for \(a=1,2\) and \(\{1,4\}^\triangle\) for \(a=3\).
	Every member of \(\sS(\balpha)\setminus\{I_{\bzero}\}\) therefore contains index two or three in every block.
	The shifts \((0,-1)\) and \((0,1)\) give differences from \(I_{\bzero}\) equal to \(\{2\}^\triangle\) and \(\{3\}^\triangle\), respectively.

	At a point of \(\cX_{12}\), only coordinates two and five are integral, and both belong to \(\{0,1,2\}\), so \(I_{\bzero}=\varnothing\).
	The shifts \((0,-1)\) and \((0,3)\) realize \(\{2\}^\triangle\) and \(\{5\}^\triangle\), respectively.
	If a shift preserves the empty negative support, then \(7u\leq\zeta_2+3\zeta_5\leq6\), so \(u\leq0\), and \(\zeta_2+2u+3v\geq0\) then gives \(v\geq0\).

	At the point of \(\cX_{23}\), the coordinates are \((-5/2,1/2,-1,0,1)\), so \(I_{\bzero}=\{3\}^\triangle\).
	If indices four and five are absent from the shifted negative support, then \(k=u+v\geq0\) and \(1-2u-k\geq0\).
	These inequalities give \(u\leq0\) and \(v\geq0\), and the third coordinate \(-1-3k\) is negative, so the negative support is exactly \(I_{\bzero}\).
	Every member of \(\sS(\balpha)\setminus\{I_{\bzero}\}\) therefore contains index four or five in every block.
	The shifts \((0,-1)\) and \((1,0)\) give differences from \(I_{\bzero}\) equal to \(\{4\}^\triangle\) and \(\{5\}^\triangle\), respectively.

	In all cases, no member of \(\sS(\balpha)\) is a proper subset of \(I_{\bzero}\), and every shift preserving \(I_{\bzero}\) has nonnegative \(\bw_q^\triangle\)-weight.
	Thus \(\cN=\{I_{\bzero}\}\) is an ordered negative support family, with \(K_{\cN}=I_{\bzero}\), \(e=1\), and \(M_{\cN}^0=\langle1\rangle\).
	The preceding shifts and the definition of \(P_{\cN}^0(\bt)\) give the monomial ideals in the following table.
	Removing the factors that are units in the local ring from the three distractions gives the last column.
	\[
		\begin{array}{c|c|c|c}
			(p,t)&I_{\bzero}&P_{\{I_{\bzero}\}}^0(\bt)&P_{B_q^\triangle}(\balpha)\\ \hline
			(-a,2a/3),\ a=1,2&\{1\}^\triangle&\langle T_2,T_3\rangle&\langle\lambda_1^q,\lambda_3^q\rangle\\
			(-3,2)&\{1,4\}^\triangle&\langle T_2,T_3\rangle&\langle\lambda_1^q,\lambda_3^q\rangle\\
			(p,t)\in\cX_{12}&\varnothing&\langle T_2,T_5\rangle&\langle\lambda_1^q,\lambda_2^q\rangle\\
			(-1/2,1/2)&\{3\}^\triangle&\langle T_4,T_5\rangle&\langle\lambda_3^q,\lambda_2^q\rangle
		\end{array}
	\]
	Since \(e=1\) and \(M_{\cN}^0=\langle1\rangle\), we have
	\[
		J_{\balpha,\cN}^{\amb}
		=
		J_{\balpha,\cN}^{\intr}
		=
		\Phi_0(P_{\cN}^0(\bt))
		=
		P_{B_q^\triangle}(\balpha).
	\]
	Each ideal in the last column is a homogeneous complete intersection of two forms of degree \(q\), so its quotient has dimension \(q^2\).
	The input \(1\), with \(m_{\balpha,\cN}=1\), produces by intrinsic perturbation a series whose coefficient at \(\bx^{\balpha}\) is \(1\); see Fact~\ref{fact:defect-coefficient-duality} and \cite[Theorem~3.2]{OS25}.
	By Proposition~\ref{prop:canonical-nilsson-identification}, this series belongs to \(\cV_{\bw_q^\triangle}\).
	The series is supported in \(\balpha+L\), and every occurring shift has negative support \(I_{\bzero}\) and nonnegative relative weight.
	Hence \(1\in\cE_{\balpha}\), so \(\balpha\in\sE_{\bbeta_q^\triangle,\bw_q^\triangle}\).
	We can now apply Proposition~\ref{prop:all-ordered-local-quotient} and \eqref{eq:actual-fake-local-inclusions} to obtain
	\[
		P_{B_q^\triangle}(\balpha)^\perp
		=
		V_{\balpha,\cN}^{\intr}
		\subseteq
		\cI_{\balpha}^{L}
		\subseteq
		\cE_{\balpha}
		\subseteq
		P_{B_q^\triangle}(\balpha)^\perp.
	\]
	Thus \(\cE_{\balpha}=\cI_{\balpha}^{L}=P_{B_q^\triangle}(\balpha)^\perp\), \(\dim_{\C}\cE_{\balpha}=q^2\), and \(\delta_{\balpha}^{L}=0\) at all ten points.

	The only integral parameter points among the eleven points are \((0,0)\) and \((-3,2)\).
	Since the columns of \(B_q^\triangle\) form a \(\Z\)-basis of \(L\), these are the only fake exponents in \(\bv_q^\triangle+L\).
	The second exponent has relative weight \((-1,\vartheta)\cdot(-3,2)=3+2\vartheta>0\), so the class of \(\bv_q^\triangle\) is least in its \(L\)-coset.
	The irrationality of \(\vartheta\) makes every class in \(\sP_{\bbeta_q^\triangle,\bw_q^\triangle}\) a singleton.
	By \eqref{eq:fixed-rank-full-local-ideal}, we have \(\dim_{\C}\cE_{\bv_q^\triangle}=2q^2\).
	All eleven fake exponents are now exponents, and \eqref{eq:formal-filtration-dimensions} gives \(\dim_{\C}\cV_{\bw_q^\triangle}=2q^2+10q^2=12q^2\).
	The lower and upper bounds in Theorem~\ref{thm:formal-solution-codimension-bounds} are both \(\lceil3q^2/4\rceil\), which proves \eqref{eq:fixed-rank-formal-solution-codimension}.
	For every \(C\) and \(k\geq2\), we have \(\delta_{P_k}^{L}=0\), so Corollary~\ref{cor:formal-solution-exact-codimension} gives \(\tau_{C,k}=0\); the term \(\tau_{C,1}\) is zero by definition.
	Theorem~\ref{thm:analytic-realization} identifies the formal spaces and their quotient with the corresponding holomorphic spaces on a suitable open set \(\Omega_q\), which proves \eqref{eq:fixed-rank-holomorphic-solution-codimension}.
\end{proof}
By Theorem~6.2 of \cite{Nak26} and Theorem~\ref{thm:fixed-rank-unbounded-solution-codimension}, we obtain the following corollary.

\begin{corollary}
	\label{cor:no-rank-only-obstruction-bound}
	Even among homogeneous \(A\)-hypergeometric systems for which the affine semigroup \(\cC(\bw)\) is normal, the column matroid is connected, and the obstruction module for an ordered negative support family has finite length, the dimension of that module admits no upper bound that depends only on \(\rank L\).
	After all exponents occurring in the canonical formal solution space and all ordered negative support families are included, the codimension of the resulting subspace of the holomorphic solution space likewise admits no upper bound that depends only on \(\rank L\).
	More precisely, the family in Theorem~6.2 of \cite{Nak26} satisfies
	\[
		\begin{aligned}
			\rank L&=2,\\
			\dim_{\C}\fD_{\cN_q^{\triangle,\circ}}(e_q^\triangle)&=q^2,\\
			\dim_{\C}\Sol_{\Omega_q}(H_{A_q^\triangle}(\bbeta_q^\triangle))
			&=12q^2,\\
			\dim_{\C}
			\frac{\Sol_{\Omega_q}(H_{A_q^\triangle}(\bbeta_q^\triangle))}
			{\cS_{\bw_q^\triangle}^{L}(\Omega_q)}
			&=\left\lceil\frac{3q^2}{4}\right\rceil.
		\end{aligned}
	\]
	For this family, the ratio of the codimension to the dimension of the holomorphic solution space satisfies
	\[
		\begin{aligned}
			\frac{\dim_{\C}\bigl(\Sol_{\Omega_q}(H_{A_q^\triangle}(\bbeta_q^\triangle))/\cS_{\bw_q^\triangle}^{L}(\Omega_q)\bigr)}
			{\dim_{\C}\Sol_{\Omega_q}(H_{A_q^\triangle}(\bbeta_q^\triangle))}
			&=\frac{\lceil3q^2/4\rceil}{12q^2}\\
			&=
			\begin{cases}
				1/16,&q\equiv0\pmod2,\\
				1/16+1/(48q^2),&q\equiv1\pmod2,
			\end{cases}\\
			&\longrightarrow\frac1{16}\qquad(q\to\infty).
		\end{aligned}
	\]
\end{corollary}
\begin{proof}
	The lattice rank is 2 for every \(q\), and Theorem~6.2 of \cite{Nak26} and Theorem~\ref{thm:fixed-rank-unbounded-solution-codimension} give the displayed dimensions, which tend to infinity with \(q\).
	The same dimensions give the displayed ratio and its limit.
\end{proof}
\section*{Acknowledgments}

This work was supported by JST SPRING, Grant Number JPMJSP2119.
Large language models were used in the preparation of this paper.
Claude was used to revise the exposition.
Every mathematical statement was checked by the author against the definitions and the cited results, and the author is responsible for the content.
The ancillary scripts reproduce selected computations in exact arithmetic in the environment specified in the ancillary README, which describes the scope of the independent checks in SageMath and Macaulay2.

\end{document}